\documentclass[12pt]{iopart}
\usepackage{amssymb,amsthm}

\usepackage{setstack}
\usepackage{graphicx}
\usepackage{float}
\usepackage{booktabs}
\usepackage{subcaption}

\newcommand{\var}{\mathrm{Var}}

\newtheorem{theorem}{Theorem}

\begin{document}

\title{Sampling Distributions as Regularization in Learned Inverse Problems}

\author{Sandra R. Babyale$^1$, and Jodi Mead$^2$}

\address{$^1$ Computing PhD, Boise State University, Boise, ID, USA}
\address{$^2$ Department of Mathematics, Boise State University, Boise, ID, USA}

\ead{sandrababyale@u.boisestate.edu}

\begin{abstract}
Neural networks have emerged as effective tools in  
solving ill-posed inverse problems. In many scientific applications, however, observational training data are not available in sufficient quantity to train a neural network model.  Subsequently,  learned inverse operators are trained on synthetic data generated from the forward model.  This requires specifying unknown parameters in the forward model, i.e.\ the target for the neural network model, and solving the forward model for synthetic, observational data i.e.\ the features for the neural network model.  Typically the unknown parameters are randomly sampled from a prescribed probability distribution.  Here we explain that the sampling strategy is not a neutral preprocessing step, rather the choice of prior sampling distribution 
specifies an implicit regularization operator. 
This result is based on the fact that the learned inverse operator minimizes empirical risk, and the well-known result that the conditional
expectation minimizes mean-square error.  We present theoretical results for the implicit regularization operator in both the infinite and finite training data settings, and with Physics Informed Neural Networks (PINN). 
These results are demonstrated numerically on 
three inverse problems of increasing complexity: a 1D linear Fredholm integral equation, a 1D nonlinear subsurface interface inversion, and a 2D nonlinear cross-well seismic traveltime tomography problem. Across all three problems, three distinct sources of regularization are identified in the learned operator, namely prior sampling, architectural, and physics-informed regularization.  It is observed that a mismatched sampling distribution degrades reconstruction quality in ways that neither a more expressive architecture nor an augmented physics residual can fully correct. Thus the
distributional shape of the sampling measure is found to be the dominant determinant of reconstruction quality.  
These results demonstrate that not only should the sampling distribution be chosen with the same care as a regularization functional, but also give a straightforward approach to implement more complicated regularization operators (e.g.\ $L_1$) using neural networks.  \end{abstract}

\section{Introduction}
\label{intro}
Inverse problems arise in many areas of science and engineering, where one seeks to recover unknown model parameters from indirect and often noisy observations. A common formulation is
$$d=G(m)+\varepsilon . $$
where $G$ is a forward operator mapping parameters $m$ to data $d$, and $\varepsilon$ represents measurement noise. Such problems are frequently ill-posed in the sense that solutions may not exist, may not be unique, or may be unstable with respect to perturbations in the data \cite{Stuart2010}.

Classical approaches address ill-posedness through regularization \cite{Stuart2010,DashtiStuart2017}. In variational formulations, one seeks a solution by minimizing a combination of data misfit and a regularization functional. In the Bayesian framework, prior information is incorporated through a probability measure on the parameter space, and the solution is characterized in terms of the posterior distribution. It is well known that these perspectives are closely related: for suitable choices of likelihood and prior, maximum a posteriori (MAP) estimation coincides with variational regularization \cite{Stuart2010}.

In recent years, there has been increasing interest in the use of machine learning for solving inverse problems \cite{Arridge2019}. In many applications, neural networks are trained to approximate an inverse operator by learning from pairs of parameters and corresponding data \cite{Adler2018}. These pairs are typically generated synthetically by sampling parameters from a prescribed distribution and evaluating the forward model. While such approaches have demonstrated empirical success, their connection to classical regularization theory is not always made explicit.

The purpose of this work is to establish a direct connection between the choice of training data in learned inversion and the selection of a regularization strategy. We show that the sampling distribution used to generate training data plays the same role as a prior in Bayesian inversion, and therefore defines an implicit regularization functional. In particular, when a neural network is trained to minimize mean squared error over data generated from a measure $\mu$, the resulting operator converges, in the infinite data limit, to the conditional expectation $\mathbb E_\mu \left[m\mid d\right]$ \cite{Wainwright2019}. This is precisely the Bayes estimator associated with the prior $\mu$ . Consequently, the design of training data is mathematically equivalent to the choice of prior or regularization functional.

This perspective provides a unified interpretation of several approaches to inverse problems. Classical Tikhonov, sparsity-promoting, and total variation regularization correspond to Gaussian, Laplace, and structured priors, respectively. In the learned setting, these same priors arise through the choice of sampling measure used to generate training data. Thus, rather than viewing learning-based inversion as a departure from classical methods, it may be understood as an alternative implementation of regularization in which the inverse operator is learned in advance.  

From a computational perspective, this viewpoint is particularly relevant for nonlinear inverse problems, where classical approaches can be challenging to apply in practice. Variational methods based on Newton or Gauss-Newton iterations require repeated evaluation of the forward operator and its Jacobian, and may suffer from instability in the presence of ill-posedness or nonsmooth regularization. Bayesian approaches based on sampling, such as Markov chain Monte Carlo, provide a principled framework for uncertainty quantification, but are often computationally prohibitive in high-dimensional settings \cite{Stuart2010,DashtiStuart2017}. 

In contrast, the learning-based approach shifts the computational burden to an offline stage in which training data are generated by solving the forward problem. Once trained, the inverse operator can be evaluated at negligible cost, without requiring iterative optimization or repeated forward solves. While the total computational cost depends on the complexity of the forward model and the amount of training data required, this amortization can be advantageous when multiple inverse problems must be solved.

An additional practical advantage arises in the treatment of nonsmooth or sparsity-promoting regularization. In classical formulations, incorporating $L_1$ or related penalties into nonlinear inverse problems leads to difficult optimization problems, for which robust and scalable algorithms are limited. In the present framework, such regularization can be implemented directly through the choice of sampling measure $\mu$, for example by drawing training parameters from distributions that promote sparsity or structured features. This suggests a flexible mechanism for incorporating prior information, even in settings where traditional regularization methods are difficult to apply.

Finally, this perspective highlights the importance of the sampling distribution itself. In particular, choosing $\mu$ to be uniform over a large parameter space does not impose meaningful regularization, and may lead to unstable or poorly conditioned learned operators in ill-posed settings. Careful design of $\mu$ is therefore essential, both for stability and for incorporating domain-specific prior knowledge.

We further analyze the role of finite training data and neural network approximation in this framework. A bias–variance decomposition is derived, identifying approximation error due to the choice of sampling measure and network architecture, estimation error due to finite data, and irreducible error arising from the non-uniqueness of the inverse problem. This decomposition clarifies the conditions under which a learned inverse operator provides a stable and accurate approximation.

Finally, we consider extensions in which additional physical constraints are incorporated into the learning objective. In particular, we show that physics-informed training can be interpreted as modifying the effective prior through a data-dependent reweighting, leading to a tilted measure that combines prior information with consistency constraints.

The paper is organized as follows. Section \ref{sec:illposed} reviews classical formulations of regularization and their equivalence to Bayesian estimation. Section \ref{sec:nn_inverse} develops the learned inverse operator framework and establishes the connection between sampling measures and regularization. Section \ref{convergence} analyzes convergence and sources of error. Numerical experiments in Section \ref{experiments} illustrate how different sampling strategies reproduce classical regularization behavior across linear and nonlinear inverse problems. Section~\ref{discussion} concludes with a discussion of the results and directions for future work.

\subsection*{Contributions}

The main contributions of this work are as follows:

\begin{itemize}

\item We establish a direct connection between learned inverse operators and classical regularization theory. In particular, we show that the sampling measure used to generate synthetic training data defines an implicit prior, and hence a regularization functional.

\item We prove that, in the infinite-data limit, a neural network trained to minimize mean squared error converges to the conditional expectation $\mathbb{E}[m \mid d]$, i.e. the Bayes estimator associated with the sampling measure. This provides a precise interpretation of learned inversion as Bayesian estimation.

\item We show that incorporating physics-based consistency terms into the training objective can be interpreted as inducing a data-dependent reweighting of the prior, leading to a tilted measure that modifies the effective regularization.

\item We derive a bias--variance decomposition for learned inverse operators, identifying approximation error due to the choice of sampling measure and network class, estimation error due to finite training data, and irreducible error arising from non-uniqueness of the inverse problem.

\item We demonstrate through numerical experiments that classical regularization strategies, including Tikhonov, sparsity-promoting, and total variation methods, can be recovered through appropriate choices of the sampling distribution used in training.

\end{itemize}

\subsection*{Related work}

There is a substantial body of work on incorporating learning into inverse problems. 
A comprehensive overview is given in \cite{Arridge2019}, which surveys a range of data-driven approaches including learned iterative methods, variational networks, and generative models.

A number of recent approaches to inverse problems combine learning with classical regularization, most notably plug-and-play priors and learned iterative schemes. In plug-and-play methods, a denoising operator is inserted into an iterative algorithm, implicitly defining a prior through its action on the parameter space. Learned iterative methods similarly unroll classical optimization algorithms and replace components such as gradients or proximal maps with neural networks \cite{Adler2018, Hauptmann2018}. In both cases, regularization is introduced through the learned operator, but remains embedded within an iterative reconstruction procedure. In contrast, the approach considered here learns the inverse operator  without iteration. The regularization is not imposed through a proximal map or denoiser, but instead through the sampling measure $\mu$ used to generate training data. As shown in Theorem 1, this induces a conditional expectation that plays the same role as a Bayesian estimator with prior $\mu$. Thus, rather than modifying an optimization algorithm, the learned operator replaces it entirely, while retaining a clear interpretation in terms of prior-driven regularization.


Operator learning provides an alternative perspective in which neural networks are trained to approximate mappings between function spaces. Notable examples include neural operator architectures that learn solution operators for partial differential equations \cite{Kovachki2021}. These approaches emphasize the representation of infinite-dimensional mappings, rather than the incorporation of explicit regularization. In contrast, the present work focuses on the role of the sampling measure in determining the stability and bias of the learned inverse operator.


More recently, generative models and score-based methods have been used to represent priors and perform posterior sampling for inverse problems 
\cite{Bora2017,Song2021}. These approaches aim to approximate the full posterior distribution, rather than a single estimator.


Finally, the role of priors in determining well-posedness and posterior behavior has been extensively studied in the Bayesian inverse problems literature
\cite{Stuart2010,DashtiStuart2017}. The present work shows that these same considerations arise naturally in the design of training data for learned inversion.

\section{Regularization for Ill-posed Inverse Problems}
\label{sec:illposed}
Consider the inverse problem 
\[d=G(m) + \varepsilon , \ \ \ 
G :  \mathcal{M} \to \mathcal{D}, \ \ \ \varepsilon \sim \nu
\]
where 
 $\mathcal{M}$ and $\mathcal{D}$ are separable Banach spaces and $G$ a non-injective linear or   nonlinear forward operator. The inverse problem can be ill-posed due to existence, uniqueness or stability of the solution.  In this section we will discuss three viewpoints to making the problem well-posed: variational regularization, Maximum A posteriori Estimation (MAP), and weak constraint with Lagrange multipliers.

For variational regularization we will consider a data misfit in an arbitrary normed vector space with convex, lower semicontinuous regularization functional
 $R : \mathcal{M} \to [0,\infty]$
 \begin{equation}
\Phi(m) = \|G(m)-d\| + \lambda R(m),
\label{eq:opt_general}
\end{equation}
and regularization parameter $\lambda > 0$. 
The optimal solution of the ill-posed problem is $\arg\min_{m \in \mathcal{M}} \Phi(m)$.  The regularization parameter $\lambda$  controls stability and bias, and must be chosen to ensure well-posedness.  In this formulation it is often chosen as small as possible so as not to change the problem too much.

We now take the Bayesian view of solving ill-posed problems.  Let $l(m;d)=\|G(m)-d\|$ define the loss and assume the likelihood
\[
L(d|m) \propto \exp(-l(m;d)).
\]
If the prior density 
\[
p(m) \propto \exp(-\lambda R(m)),
\]
then the posterior density is 
\[
p(m|d) 
\propto \exp\left(
-l(m;d) - \lambda R(m)
\right).
\]
The maximum a posteriori (MAP) estimator therefore satisfies
\begin{equation}
m_{\mathrm{MAP}}
\in \arg\min_{m \in \mathcal{M}}
\left\{
\|G(m)-d\| + \lambda R(m)
\right\}.
\label{eq:map_equivalence}
\end{equation}
Hence, variational regularization and MAP estimation can be viewed as the same optimization problem.  However, in MAP estimation the regularization term represents prior information and the regularization parameter reflects prior uncertainty.

Consider now the weakly constrained  problem
\begin{equation}
\min_{m \in \mathcal{M}} \ \|G(m)-d\|
\quad \text{subject to} \quad C(m)\leq \delta,
\label{eq:constrained_problem}
\end{equation}
where $C:=\mathcal{M}\to \mathcal{Z}$ encodes structural constraints.  These constraints could incorporate physical laws, or specify bounds on the parameters \cite{Mead_Renaut_2010}.  We incorporate the constraint by adding an extra term
into the objective function and introduce a Lagrange multiplier $\lambda_L$.  If the constraint is enforced quadratically we have 
\begin{equation}
\min_{m \in \mathcal{M}}
\left\{
\|G(m)-d\|+ \frac{\lambda_L}{2}\|C(m)\|_{\mathcal{Z}}^2
\right\}.
\label{eq:penalty}
\end{equation}
This effectively circumscribes an ellipsoid constraint around the space of possible values of $m$ defined by $C(m)$. 
In this case, weak enforcement of constraints via Lagrange penalties is equivalent to regularization with
\[
R(m) = \frac12 \|C(m)\|_{\mathcal{Z}}^2.
\]
Consequently, variational regularization, MAP estimation, and weak constraint enforcement are mathematically equivalent within a unified variational framework. Table \ref{table_1} shows examples of prior distributions for MAP estimates and their corresponding regularization functionals. 

\begin{table}[h]
\begin{center}
\begin{tabular}{|l|c|c|c|c|}\hline
Prior Distribution & Gaussian & Laplace & Total Variation& Uniform \\ \hline
Regularization Operator & $\|m\|_2^2$ & $\| m \|_1 $ & $\| \nabla m \|_1 $ & None  \\ \hline
\end{tabular}
\caption{MAP and Regularization equivalences}\label{table_1}
\end{center} 
\end{table}

\section{Learning an Inverse Operator via Neural Networks} \label{tb:equivalences}
\label{sec:nn_inverse}

We now consider the construction of an approximate inverse operator using supervised learning. The focus here is underdetermined problems with a single realization of observational data $d_{obs}$.  Training data is synthetically generated from the forward operator $G(m)$. This requires choosing 
parameter values $m_{i}$ and solving the forward model for $d_{i}$. We will show that the sampling distribution used to generate training data defines a regularization functional. 

\section*{Synthetic Training Data}

Define
 $\mu$ a probability measure on $\mathcal{M}$. We generate training pairs
\[
m_{i} \sim \mu,
\qquad
d_{i} = G(m_{i}) + \varepsilon_i.
\] 
The data  are a realization of the random variable $G(m_{i}) + \varepsilon_i$.  The sampling distribution $\mu$ induces a joint distribution $\rho(m,d)$ through $G$, i.e.\ $\rho(m,d)=\mu(m)\nu(d-G(m))$.
As noted in \cite{latz_2020}  if $m_i$ and $\varepsilon$ are independent variables defined on an underlying probability space $(\Omega ,A , P)$ then ${d_i} \in A$.    For example, if we assume that $\varepsilon_i \sim \mathcal{N} (0,C)$ with  $m_{i} \sim \mu$ independent, then $d_i\sim \mathcal{N} (G(m_i),cov(G(m_i))+C)$.  

The next step is to train a neural network
\[
G^\dagger_\theta : \mathcal{D} \to \mathcal{M}.
\]
In practice, the inverse operator $G^\dagger_\theta$ is restricted to
a \emph{hypothesis class} $\mathcal{H}$, which is the set of functions
representable by a neural network with a fixed architecture and
parameter space $\Theta$:
\[
\mathcal{H} = \left\{G^\dagger_\theta : \mathcal{D} \to \mathcal{M}
\;\middle|\; \theta \in \Theta \right\}.
\]
The choice of architecture determines $\mathcal{H}$ and hence the
expressive capacity of the learned operator. We train the network by
minimizing the empirical risk over $\mathcal{H}$:
\begin{equation} \label{eq:empirical_risk}
\mathcal{J}_K(\theta) = \frac{1}{K}\sum_{i=1}^K
\|G^\dagger_\theta(d_i) - m_i\|^2, \quad \theta \in \Theta.
\end{equation}
%
The optimal inverse estimate from the neural network is then \[m_{pred}=G^\dagger_\theta(d_{obs}).\]

Bayesian methods also seek to minimize the expected risk. This is not done for all possible parameters, rather for those
 drawn at random according to the prior probability. Similarly, $G^\dagger_\theta$ is found by minimizing (\ref{eq:empirical_risk}) over the training data generated by $m \sim \mu$ and in finite dimensions
 $\{d_i,m_i\}_{i=1}^K$.
The difference between the Bayes estimator and the neural network is the former minimizes expected loss with respect to $m$ and the latter
with respect to $\theta$.
As $K \to \infty$, the empirical risk converges to the population functional
\[ \mathcal R(d)=\mathbb{E}_{\mu} \|G^\dagger_\theta(d)-m\|^2  \]

The following result in statistical decision theory \cite{Wainwright2019} shows that the training data distribution in learned inversion is mathematically equivalent to the regularization functional in classical inverse problems.

\begin{theorem}\label{thm:samp_reg}
Let $d=G(m)+\varepsilon$, with $m\sim \mu$ and $\mathbb E \| m \|^2 < \infty$.  Then the minimizer of
\[
\mathcal R(d)=\mathbb E_{\mu} \|G_\theta^\dagger(d)-m\|^2
\]
over $\theta$ is
\[
G^\dagger_*(d)=\mathbb E_\mu[m\mid d].
\]
Moreover the minimum risk is 
\[
\mathcal R_*(d)
=
{\var}(m\mid d).
\]
Consequently, the training measure $\mu$ determines the regularizing mapping.
\end{theorem}
\begin{proof}
We fix $d$, then insert and subtract the conditional mean,
\begin{eqnarray*}\|G^\dagger_\theta(d)-m\|^2&=&\|G^\dagger_\theta(d)-\mathbb E[m \mid d]\|^2 + \|\mathbb E[m \mid d]-m\|^2\\ && +2\langle G^\dagger_\theta(d)-\mathbb E[m \mid d], \mathbb E[m \mid d]-m\rangle.
\end{eqnarray*}
Now take the conditional expectation in $m\mid d$, then
\[\mathbb E\left[ \langle G^\dagger_\theta(d)-\mathbb E[m \mid d], \mathbb E[m \mid d]-m\rangle\right] = \langle G^\dagger_\theta(d)-\mathbb E[m \mid d], \mathbb E[m \mid d]-\mathbb E[m \mid d]\rangle.\]
The cross term vanishes and 
\[
\mathbb E_\mu\left[ \|G^\dagger_\theta(d)-m\|^2 \mid d\right]
=
\|G^\dagger_\theta(d)-\mathbb E_\mu[m \mid d]\|^2
+ \var (m\mid d).\]
Now integrate over $d$ and minimize over $\theta$. The second term does not depend on $\theta$ and the first term is minimized by $G^\dagger_\theta(d)=\mathbb E_\mu[m \mid d]$.
\end{proof}

In Bayesian inversion, the prior $\mu$ appears explicitly. In neural inversion, the $\mu$ appears implicitly through the expectation defining risk.
Yet in both cases the same conditional expectation arises. Thus sampling from $\mu$ to create training data embeds the prior directly into the learning objective.  For example, if  $\mu\sim \mathcal N(m_0,1/ \lambda)$ then this parallels Tikhonov regularization that minimizes
\[\Phi(m) = \|G(m)-d\|_2^2 + \lambda \|m-m_0\|_2^2.\]
Here, $\lambda$ controls how strongly prior information $m_0$ influences the solution.  

Even when the Bayes inverse is attained, the risk is not zero, it is the posterior variance averaged over the data distribution. The posterior variance represents  uncertainty arising from non-injectivity of $G$. If $G$ is injective and noise-free on the support of $\mu$, the posterior variance term vanishes. Otherwise, it persists no matter how accurate the network may be. If we restrict  $\mu$ to a narrower subset of $\mathcal M$ the posterior variance is reduced, but we increase the bias outside that subset.  This is similar to the bias and stability tradeoff in regularization theory.
	
Regularization makes the problem well-posed by restricting the possible solutions. For the neural network inverse operator, the choice of $\mu$ restricts the solution manifold.  Thus, sampling strategy plays the same structural role as regularization.  Figure \ref{fig:regularization} illustrates how the choice of $\mu$ or regularization operator restricts the 
geometry of the solution in the case of $L_2$ regularization ($\mu$ Gaussian) or $L_1$ regularization ($\mu$ Laplace).
\begin{figure}[h]
\centering
\includegraphics[width=.40\textwidth]{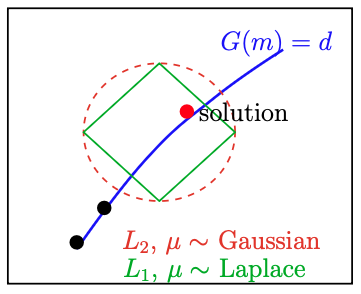}
\caption{Regularization by prior geometry in an ill-posed inverse problem. The points on the curve $G(m)=d$ represent non-unique solutions. The prior or constraint induces a unique solution.}
\label{fig:regularization}
\end{figure}

\section*{Physics Informed Networks (PINN)}
Physics informed Networks (PINN) were introduced in \cite{raissi2019physicsinformed} and have been used extensively, including for inverse problems in \cite{perezbernal2025physicsinformedneuralnetworkssolve,Jagtap2022PhysicsinformedNN}.   PINNs optimize a cost function that adds information about the physics, in this case $d=G(m)$, to the risk.  The neural network parameters $\theta$ are found by minimizing a regularized empirical risk.
As $K \to \infty$, this empirical risk converges to the population functional
\[ \mathcal R_\alpha(d)=\mathbb{E}_{\mu} \left[ \|G^\dagger_\theta(d)-m\|^2 + \alpha  \| G(G^\dagger_\theta(d))-d\|^2 \right]. \]

\begin{theorem} \label{thm_PINN}
Let $d = G(m) + \varepsilon$, with $m \sim \mu$ and
$\mathbb{E}\|m\|^2 < \infty$, and assume $\alpha > 0$.
Define the tilted measure $\mu_\alpha$ with density
\[
d\mu_\alpha \propto e^{-\alpha\|G(m)-d\|^2} \, \mu(dm).
\]
Then the minimizer of
\[
\mathcal{R}_\alpha(d)
= \mathbb{E}_{\mu}\!\left[\|G^\dagger_\theta(d) - m\|^2
+ \alpha\|G(G^\dagger_\theta(d)) - d\|^2\right]
\]
over $\theta$ is
\[
G^\dagger_{\alpha *}(d) = \mathbb{E}_{\mu_\alpha}[m \mid d].
\]
Moreover the minimum risk is
\[
\mathcal{R}_{\alpha *}(d) = \mathrm{Var}_{\mu_\alpha}(m \mid d).
\]
\end{theorem}
\begin{proof}
The second term in the risk, $\alpha\|G(G^\dagger_\theta(d)) - d\|^2$,
depends on $d$ and on the network output $G^\dagger_\theta(d)$, but not
on the integration variable $m$. Therefore
\begin{eqnarray*}
\mathcal{R}_\alpha(d)
&=& \mathbb{E}_\mu\!\left[\|G^\dagger_\theta(d) - m\|^2\right]
+ \alpha\|G(G^\dagger_\theta(d)) - d\|^2 \\
&=& \int \left[\|G^\dagger_\theta(d) - m\|^2
+ \alpha\|G(G^\dagger_\theta(d)) - d\|^2\right] \mu(dm),
\end{eqnarray*}
We claim that minimizing $\mathcal{R}_\alpha(d)$ over $\theta$ is
equivalent to minimizing
$\mathbb{E}_{\mu_\alpha}[\|G^\dagger_\theta(d)-m\|^2]$ over $\theta$.
To see this, let
\[
Z = \int e^{-\alpha\|G(m)-d\|^2}\,\mu(dm)
\]
be the normalizing constant of $\mu_\alpha$. Since $\mu_\alpha$ is
absolutely continuous with respect to $\mu$ with Radon-Nikodym
derivative
\[
\frac{d\mu_\alpha}{d\mu}(m)
= \frac{e^{-\alpha\|G(m)-d\|^2}}{Z},
\]
the change of measure formula gives
\begin{eqnarray*}
\mathbb{E}_{\mu_\alpha}\!\left[\|G^\dagger_\theta(d) - m\|^2\right]
&=&\int \|G^\dagger_\theta(d) - m\|^2\,\mu_\alpha(dm)\\
&=& \frac{1}{Z}\int \|G^\dagger_\theta(d) - m\|^2\,
e^{-\alpha\|G(m)-d\|^2}\,\mu(dm).
\end{eqnarray*}
The key observation is that the additive penalty
$\alpha\|G(G^\dagger_\theta(d))-d\|^2$ in $\mathcal{R}_\alpha(d)$
and the multiplicative likelihood weight $e^{-\alpha\|G(m)-d\|^2}$
in $\mu_\alpha$ arise from the same Gaussian likelihood, and both
are minimized when $G^\dagger_\theta(d)$ is consistent with the data
$d$ through the forward operator $G$. Specifically, for fixed $d$,
\[
\mathcal{R}_\alpha(d)
= \int \|G^\dagger_\theta(d)-m\|^2\,\mu(dm)
+ \alpha\|G(G^\dagger_\theta(d))-d\|^2
\]
and
\[
Z\cdot\mathbb{E}_{\mu_\alpha}\!\left[\|G^\dagger_\theta(d)-m\|^2\right]
= \int \|G^\dagger_\theta(d)-m\|^2\,
e^{-\alpha\|G(m)-d\|^2}\,\mu(dm).
\]
Both expressions are minimized by the same $\theta$: the additive
term $\alpha\|G(G^\dagger_\theta(d))-d\|^2$ and the likelihood weight
$e^{-\alpha\|G(m)-d\|^2}$ penalize the same discrepancy between the
forward model evaluated at the network output and the observed data,
so they share the same minimizer. Since $Z$ does not depend on
$\theta$, minimizing $\mathcal{R}_\alpha(d)$ over $\theta$ is
equivalent to minimizing
\[
\mathbb{E}_{\mu_\alpha}\!\left[\|G^\dagger_\theta(d) - m\|^2\right]
\]
over $\theta$. The result then follows from
Theorem~\ref{thm:samp_reg} with $\mu_\alpha$ in place of $\mu$.
\end{proof}

For PINN, the tilted measure $\mu_\alpha$ is now the effective prior. Thus $\alpha$ may be interpreted as inducing additional regularization through a data-dependent re-weighting of the prior.

\section*{Finite Training Data}
In the previous sections, we derived results at the population level over $\mu$.  Here 
we address the  finite-sample reality of training. The following theorem addresses what happens when there are only finitely many samples.  The subsequent decomposition follows from  bias-variance arguments in statistical learning theory (see, e.g., \cite{hastie01statisticallearning,Wainwright2019}), specialized here to the inverse problem setting.  
\begin{theorem}\label{thm:finite}
Let $S=\{(m_i,d_i)\}_{i=1}^K$, $m_i \sim \mu$,
$d_i = G(m_i) + \varepsilon_i$, and $\varepsilon_i \sim \nu$.
Let $\theta_S$ denote the parameters learned from dataset $S$, and
define the inverse operator
\[
G^\dagger_{\theta_S} : \mathcal{D} \to \mathcal{M}.
\]
The minimizer of
\[
\mathcal{R}_S(d) = \mathbb{E}_{S}\left[
\|G^\dagger_{\theta_S}(d) - m\|^2\right]
\]
over $\theta$ is
\[
G^\dagger_{S*}(d) = \mathbb{E}_{S}\left[\mathbb{E}_\mu[m \mid d]\right].
\]
Moreover, the risk admits the bias-variance decomposition
\begin{eqnarray*}
\mathcal{R}_S(d)
&=& \|\mathbb{E}_S G^\dagger_{\theta_S}(d) - G^\dagger_{S*}(d)\|^2 \\
&&+\; \mathbb{E}_S\!\left[\|G^\dagger_{\theta_S}(d)
   - \mathbb{E}_S G^\dagger_{\theta_S}(d)\|^2\right] \\
&&+\; \mathbb{E}_S\!\left[\mathrm{Var}(m \mid d)\right],
\end{eqnarray*}
where the three terms are the squared bias, the variance of the
estimator over datasets $S$, and the irreducible noise, respectively.
The minimum risk is
\[
\mathcal{R}_{S*}(d) = \mathbb{E}_S\!\left[\mathrm{Var}(m \mid d)\right].
\]
\end{theorem}

\begin{proof}
We fix $d$ and apply Theorem~\ref{thm:samp_reg} pointwise in $S$.
For each fixed dataset $S$, the conditional mean $\mathbb{E}_\mu[m \mid d]$
minimizes $\mathbb{E}_\mu[\|G^\dagger_{\theta_S}(d) - m\|^2 \mid d]$
with residual $\mathrm{Var}(m \mid d)$.
Inserting and subtracting $\mathbb{E}_\mu[m \mid d]$ and taking the
expectation over $S$, the cross term vanishes since
$\mathbb{E}_S[\mathbb{E}_\mu[m\mid d] - m \mid d] = 0$,
giving
\[
\mathbb{E}_S\!\left[\|G^\dagger_{\theta_S}(d) - m\|^2\right]
= \mathbb{E}_S\!\left[\|G^\dagger_{\theta_S}(d)
  - \mathbb{E}_\mu[m \mid d]\|^2\right]
+ \mathbb{E}_S\!\left[\mathrm{Var}(m \mid d)\right].
\]
The second term does not depend on $\theta$, so the minimizer over
$\theta$ is that of the first term. We now decompose the first term
by inserting and subtracting $\mathbb{E}_S G^\dagger_{\theta_S}(d)$:
\begin{eqnarray*}
\|G^\dagger_{\theta_S}(d) - \mathbb{E}_\mu[m\mid d]\|^2
&=& \|G^\dagger_{\theta_S}(d)
    - \mathbb{E}_S G^\dagger_{\theta_S}(d)\|^2 \\
&&+\; \|\mathbb{E}_S G^\dagger_{\theta_S}(d)
    - \mathbb{E}_\mu[m \mid d]\|^2 \\
&&+\; 2\langle G^\dagger_{\theta_S}(d)
    - \mathbb{E}_S G^\dagger_{\theta_S}(d),\,
    \mathbb{E}_S G^\dagger_{\theta_S}(d)
    - \mathbb{E}_\mu[m \mid d]\rangle.
\end{eqnarray*}
Taking the expectation over $S$, the cross term vanishes since
$\mathbb{E}_S[G^\dagger_{\theta_S}(d)
- \mathbb{E}_S G^\dagger_{\theta_S}(d)] = 0$, giving the
bias-variance decomposition stated in the theorem. Both the squared
bias and the variance term are non-negative. We now argue that they
are simultaneously minimized by the same $\theta$.

The squared bias 
is zero if and only if
$\mathbb{E}_S G^\dagger_{\theta_S}(d) 
= \mathbb{E}_S[\mathbb{E}_\mu[m\mid d]]$.
The variance term 
measures the fluctuation
of $G^\dagger_{\theta_S}(d)$ around its mean over $S$.
At the minimizer 
the network output equals the deterministic quantity
$\mathbb{E}_S[\mathbb{E}_\mu[m\mid d]]$, which does not depend on $S$.
The variance term therefore also vanishes at this point.
Note that the variance term does depend on $\theta$ for a general
network, since $G^\dagger_{\theta_S}(d)$ fluctuates over $S$;
it is only at the optimal $\theta$ that it vanishes because the
output becomes deterministic. The minimum risk is therefore
\[
\mathcal{R}_{S*}(d) = \mathbb{E}_S\!\left[\mathrm{Var}(m \mid d)\right].
\]

\end{proof}

This decomposition highlights the distinct roles of the sampling measure $\mu$, the hypothesis class $\mathcal{H}$, and the training sample size. In particular, the approximation error is controlled by the choice of $\mu$, reinforcing its interpretation as a regularization mechanism.  The irreducible error corresponds to the conditional variance $\mathrm{Var}(m \mid d)$, which reflects the non-identifiability of the inverse problem \cite{Stuart2010}.

\section{Convergence and Sources of Error in NN Operator}
\label{convergence}

We now discuss the conditions under which the learned inverse operator
$G^\dagger_{\theta_S}$ provides a reliable approximation to the true solution. 
For squared error loss, the optimal predictor is given by the conditional expectation, and the excess risk can be decomposed into approximation and estimation components \cite{Wainwright2019}.  
The bias-variance decomposition in Theorem~\ref{thm:finite} identifies three
contributions to the total risk:
\begin{eqnarray*}
\mathbb{E}_S\left[\|G^\dagger_\theta(d) - m\|^2\right]
&= \underbrace{\|\mathbb{E}_S G^\dagger_\theta(d) - G^\dagger_*(d)\|^2}
             _{\text{approximation error}}
\\
&\quad
 + \underbrace{\mathbb{E}_S\|G^\dagger_\theta(d)
               - \mathbb{E}_S G^\dagger_\theta(d)\|^2}
              _{\text{estimation error}}
 \\
&\quad
 + \underbrace{\mathbb{E}_S\left[\mathrm{\mathrm{Var}}(m \mid d)\right]}
              _{\text{irreducible error}}.
\end{eqnarray*}

The approximation error measures how far the mean prediction of the network is from the
optimal operator $G^\dagger_*(d) = \mathbb{E}_{\mu}[m \mid d]$.
It arises from two distinct sources: the choice of sampling measure $\mu$, and the
expressive capacity of the hypothesis class $\mathcal{H}$ which in practice is determined by the neural
network architecture. Note that even if $\mu$ is well-chosen, the learned operator may
fail to approximate $\mathbb{E}_\mu[m \mid d]$ if the hypothesis
class $\mathcal{H}$ is insufficiently expressive.  The neural network
 may itself impose structural biases such as smoothness
or low-complexity representations.  The regularization
consequences of restricting to a particular architecture have been
studied in the deep learning literature in the context of inverse problems
\cite{ongie2020deep}. 
Here we consider three architectures in Section~5 (MLP, CNN, and DeepONet)  to span a range of
inductive biases. 


The estimation error
term measures the variability of the network's output across different
training datasets $S$ of size $K$.
It reflects the finite-sample nature of learning captured in
Theorem~\ref{thm:finite}, and with only $K$ training pairs, the empirical risk
minimizer $G^\dagger_\theta$ will fluctuate around its mean.
This term vanishes as $K \to \infty$, provided the network is trained to
convergence on each $S$.


The irreducible error term is the conditional variance $\mathrm{Var}(m \mid d)$, which
reflects the fundamental non-uniqueness of the inverse problem.  Even with
perfect knowledge of the forward operator and an infinite amount of training
data, there remains uncertainty in $m$ given $d$ that cannot be resolved.

The inverse problem is made well-posed if it satisfies existence, stability, and convergence.
In this setting, existence is guaranteed by the minimization of the empirical risk
$\mathcal{J}_K(\theta)$, which always produces a network $G^\dagger_\theta$.
Stability corresponds to control of the estimation error, and as $K$
increases, the network's predictions become less sensitive to the particular
choice of training set $S$.
Convergence corresponds to the vanishing of both approximation and estimation errors.

Convergence in this
setting requires not only that $K \to \infty$, but also that the training
distribution $\mu$ is well-specified, and the network architecture becomes sufficiently expressive.
If $\mu$ assumes very low probability near  the true model parameter, the
approximation error will not vanish even as $K \to \infty$.
Quantitative rates of convergence for the posterior measure in
infinite-dimensional models have been established by
\cite{GhosalGhoshVanDerVaart2000}. However, as emphasized in \cite{Arridge2019}, consistency may fail for natural priors even with
an infinite amount of data, meaning the prior $\mu$ may have a large and
persistent influence on the posterior. In infinite-dimensional space, both in the Bayesian setting and in neural network approximations, the prior or training data distribution is not asymptotically negligible, and it can dominate the posterior unless carefully chosen.

The tilted measure $\mu_\alpha$ of Theorem~\ref{thm_PINN} offers a
partial remedy when the prior $\mu$ is misspecified. By weighting
$\mu$ by the likelihood $e^{-\alpha\|G(m)-d\|^2}$, the effective
prior concentrates more mass on model parameters that are consistent
with the data through the forward operator. This may reduce the
approximation error~(I) relative to $\mu$ alone, provided the
physics residual is informative. However, increasing $\alpha$
is not a viable substitute for a well-chosen prior. In the presence
of noise $\varepsilon \sim \nu$, the exact constraint $G(m) = d$ is
never satisfied by the true model, so concentrating $\mu_\alpha$
too strongly on the data-consistent set amplifies the effect of
noise and may destabilize the reconstruction. This is the same
instability that motivates regularization in the first place:
the ill-posed nature of the inverse problem means that enforcing
data consistency without regularization does not yield a reliable
estimate. Thus $\alpha$ is best understood as a secondary
regularization parameter that refines a well-chosen prior $\mu$,
rather than as a mechanism for correcting prior misspecification.
This is consistent with the numerical findings of Section~5, where
the physics-informed objective improves reconstruction quality only
when the prior is already well matched to $m_\mathrm{true}$.


%
%
%
%
%
%
%
%

\section{Numerical Experiments}
\label{experiments}

We consider three inverse problems of increasing complexity. In each problem, we train $G^\dagger_\theta$ under four prior families (Gaussian, Laplace, total variation, and uniform) and three network architectures (MLP, CNN, DeepONet). For the Gaussian prior, two covariance structures are considered. We also implement PINN with 
the physics-informed objective $\mathcal{R}_\alpha$ with $\alpha = 1$. The resulting neural network models are compared to MAP estimates under a Gaussian prior, and the prior sample mean 
$\bar{m} = \frac{1}{K}\sum_{i=1}^K m_i$. 

\subsection{Test Problems}
\textbf{1D Linear Wing Problem.} The Wing problem is a discretization of the Fredholm integral equation of the first kind
\begin{equation}
\int_0^1 K(s,t)\, m(t)\, dt = g(s),
\label{eq:wing_forward}
\end{equation}
where the kernel $K(s,t) = t\,e^{-st^2}$ and the exact right-hand side is $g(s) = (e^{-s/9} - e^{-4s/9})/(2s)$. Here $m(t)$ denotes the unknown parameter function, discretized as $\mathbf{m} \in \mathbb{R}^N$.
The discretized forward operator $\mathbf{G} \in \mathbb{R}^{M \times N}$ has
entries
\begin{equation}
G_{ij} = t_j \exp(-s_i t_j^2)\,\Delta t,
\label{eq:wing_G}
\end{equation}
where $t_j$ and $s_i$ are uniform grids on $[0,1]$ with $N = 50$ and $M = 20$ points respectively, and $\Delta t = t_2 - t_1$.
The true parameter vector $\mathbf{m}_{\rm true} \in \mathbb{R}^N$ is the discretization of the discontinuous rectangular pulse
\begin{equation}
m_{\rm true}(t) = \left\{ \begin{array}{ll} 
1, & \frac{1}{3} < t < \frac{2}{3}, \\ 
0, & \text{otherwise,} 
\end{array} \right.
\label{eq:wing_true}
\end{equation} 
and the observational noise is $\varepsilon \sim \mathcal{N}(0, C_d)$ with $\sigma_d = 0.01$.

\textbf{1D Nonlinear Subsurface Interface Inversion.}
This problem estimates a subsurface interface function $z(w)$,
representing the depth of a geological anomaly, from surface measurements of the potential field. Following \cite{tarantola2005inverse}, the forward operator
evaluates
\begin{equation}
d(x) = \int_a^b \log \frac{(x-w)^2 + H^2}{(x-w)^2 +
[H - z(w)]^2}\, dw,
\label{eq:anomaly_forward}
\end{equation}
where $H = 10$ km is a fixed depth parameter and the domain is
$[a,b] = [0, 100]$ km. The interface function $z(w)$ is discretized as
$\mathbf{m} \in \mathbb{R}^N$ on a uniform grid of $N = 100$ points, and
data are collected at $M = 15$ uniformly spaced surface locations.
The true parameter vector $\mathbf{m}_{\rm true}$ is the discretization of
the smooth Gaussian-shaped anomaly
\begin{equation}
z_{\rm true}(w) = z_{\max} \exp\!\left(-\frac{5(w - w_0)^2}{N}\right),
\label{eq:anomaly_true}
\end{equation}
with peak amplitude $z_{\max} = 2.5$ km at $w_0 = 50.5$ km,
and the observation noise is $\varepsilon \sim \mathcal{N}(0, C_d)$ with $\sigma_d = 0.1$.

\textbf{2D Nonlinear Cross-Well Seismic Traveltime Tomography.}
The third problem recovers the subsurface slowness field from 
cross-well seismic traveltimes. The unknown is the slowness 
field $m(\mathbf{x}) = 1/v(\mathbf{x})$, where $v(\mathbf{x})$ 
is the seismic velocity at point $\mathbf{x}$ in the medium. 
The traveltime from a source to a receiver is governed by the 
continuous line integral
\begin{equation}
    t = \int_L m(\mathbf{x})\, dL = \int_L 
    \frac{1}{v(\mathbf{x})}\, dL,
    \label{eq:seismic_continuous}
\end{equation}
where $L$ is the ray path connecting source and receiver and 
$dL$ is an infinitesimal length element along it. Discretizing 
the domain into $N$ cells and approximating the ray path 
through each cell gives the forward operator
\begin{equation}
    t_{ij} = \sum_{k=1}^{N} L_{ijk}\, m_k,
    \label{eq:seismic_forward}
\end{equation}
where $m_k = 1/v_k$ is the slowness in the $k$th grid cell, 
$v_k$ is the corresponding velocity, and $L_{ijk}$ is the 
ray-path length through the $k$th cell for source-receiver 
pair $(i,j)$, computed by iterative ray bending with up to 
$50$ iterations and convergence tolerance $10^{-4}$ 
\cite{aster2018parameter}.
The model domain is $1600\,\text{m} \times 1600\,\text{m}$, 
discretized on a $7 \times 7$ slowness grid giving $N = 49$. 
Seven sources and seven receivers are placed along the left 
and right boundaries respectively, yielding $M = 49$ 
observations. The true velocity model is
\begin{equation}
v_{\rm true}(\mathbf{x}) = v_0\left(1 + 0.10\, e^{-\gamma\|\mathbf{x} - 
\mathbf{x}_{\rm fast}\|^2}\right)\left(1 - 0.15\, e^{-\gamma\|\mathbf{x} - 
\mathbf{x}_{\rm slow}\|^2}\right),
\label{eq:vtrue}
\end{equation}
where $v_0 = 2900$ m/s, $\gamma = 4\times 10^{-5}$ m$^{-2}$, and 
$\mathbf{x}_{\rm fast}$, $\mathbf{x}_{\rm slow}$ are the centers of the 
fast and slow anomalies respectively. The true slowness vector $\mathbf{m}_{\rm true}$ is obtained as $\mathbf{m}_{\rm true} = 1/v_{\rm true}(\mathbf{x})$, and the 
observation noise is $\varepsilon \sim \mathcal{N}(0, C_d)$ 
with $\sigma_d = 0.001$ s.

\subsection{Prior Families}
We consider four prior families over $\mathbf{m} \in \mathbb{R}^N$, each centered at the prior mean $\mathbf{m}_0$ of the respective problem and calibrated to the same marginal standard deviation $\sigma$, so that they differ only in distributional shape. This isolates the effect of prior shape on the regularization character of the learned inverse operator, as established in Section~3. The four families are defined as follows. 

The Gaussian prior is
$\mathbf{m} \sim \mathcal{N}(\mathbf{m}_0, \mathbf{C}_m)$, where the
prior covariance is given by the Gaussian covariance model
\begin{equation}
(\mathbf{C}_m)_{jk} = \sigma^2 \exp\!\left\{-\frac{\|\mathbf{x}_j -
\mathbf{x}_k\|^2}{2\Delta^2}\right\},
\label{eq:cov}
\end{equation}
where $\mathbf{x}_j$ denotes the spatial location of the $j$th model
parameter and $\Delta$ is the correlation length. This prior induces
Tikhonov-type regularization and yields smooth solutions. For comparison,
the scaled identity covariance $\mathbf{C}_m = \sigma^2 \mathbf{I}$ is
also considered, encoding no spatial correlation structure. 

The Laplace prior draws each component independently as
$\mathbf{m} \sim \mathrm{Laplace}(\mathbf{m}_0, \mathbf{b})$, where
$b_j = \sqrt{(\mathbf{C}_m)_{jj}/2}$ matches the marginal variance
to that of the Gaussian prior, inducing $\ell_1$-type regularization
and favoring sparse deviations from $\mathbf{m}_0$.

The total variation prior is generated by accumulating independent
Laplace increments along the spatial coordinate,
$m_{j+1} - m_j \sim \mathrm{Laplace}(0, b_{\rm grad})$, centred at $\mathbf{m}_0$ and rescaled to match $\sigma$. This implements the
discrete total variation functional $\sum_j |m_{j+1} - m_j|$ as the implicit regularizer, favoring piecewise-constant solutions. For the 2D problem, independent Laplace walks
are accumulated along both spatial directions and summed, implementing
the anisotropic discrete total variation
$\sum_{i,j}(|m_{i+1,j} - m_{i,j}| + |m_{i,j+1} - m_{i,j}|)$.

The uniform prior draws each component independently from
$m_j \sim \mathrm{Uniform}(m_{0,j} - a,\; m_{0,j} + a)$,
where $a = \sigma\sqrt{3}$ matches
the marginal standard deviation to the other three families. Unlike the Gaussian, Laplace,
and total variation priors, the uniform prior encodes no spatial regularization
structure. Its components are drawn independently, and the distribution
places equal probability on all values within the interval $[m_{0,j} - a,\; m_{0,j} + a]$. It therefore isolates the regularization
contributed by the network architecture and training procedure from that
contributed by the distributional shape of the prior.

The parameters for each problem are given in Table~\ref{tab:params}.

\begin{table}[htbp]
\centering
\begin{tabular}{|lcccc|}
\hline
Problem & $\sigma$ & $\mathbf{m}_0$ & $\Delta$ & $K$ \\
\hline
Wing & $1$ & $\mathbf{0}$ & $0.02$ & $100{,}000$ \\
Subsurface interface & $1$ & $\mathbf{0}$ & $1$ km & $100{,}000$ \\
Seismic tomography &~~~~ $0.05/v_0$ & ~~~~$(1/v_0)\mathbf{1}$ & ~~$400$ m & $50{,}000$ \\
\hline
\end{tabular}
\caption{Parameters for each problem. For the Laplace, TV, and
uniform priors, $\Delta$ does not apply.}
\label{tab:params}
\end{table}

\subsection{Network Architectures}
Three network architectures are used as learned inverse operators
$G^\dagger_\theta \colon \mathbb{R}^M \to \mathbb{R}^N$:
a residual multilayer perceptron (MLP), a convolutional neural network
(CNN), and a deep operator network (DeepONet) with a Fourier trunk. The
three architectures impose distinct inductive biases on the learned
inverse mapping. The MLP uses fully connected layers with no explicit
spatial structure, allowing flexible global interactions between inputs
and outputs. The CNN exploits the ordering of the observation vector
through shared convolutional kernels of multiple widths, promoting
multiscale local feature extraction. The DeepONet separates the
dependence on observations and spatial coordinates through branch and
trunk networks; the Fourier-feature trunk provides a spectral
parameterization that captures smooth, low-frequency components
efficiently. Together, these architectures enable a systematic
assessment of how architectural regularization interacts with
prior-induced regularization across problems of varying complexity
and nonlinearity.

The residual MLP uses fully connected layers with hidden units
$(128, 256, 512, 256, 128)$. A residual skip connection projects
the output of the first hidden layer to match the dimension of the
final hidden layer and adds it to the final hidden representation
prior to the output layer.

The CNN applies parallel one-dimensional convolutional branches to
the ordered observation vector, with kernel sizes $(1, 3, 5, 7)$.
Each branch comprises two Conv1D layers with $64$ filters followed
by global average pooling. The branch outputs are concatenated and
passed through dense layers of sizes $(512, 256)$ prior to the
output layer. For the seismic tomography problem, two-dimensional
convolutions are used to account for the spatial structure of the
domain.

The DeepONet with a Fourier trunk separates the dependence of the
learned operator on observations and spatial coordinates through
branch and trunk networks. The branch network encodes
$\mathbf{d} \in \mathbb{R}^M$ into a latent representation
$\mathbf{b}(\mathbf{d}) \in \mathbb{R}^p$ via a fully connected
MLP with depth $4$ and hidden width $256$. The trunk network maps
each query point $\mathbf{x}$ to sinusoidal positional features
constructed from $n_{\rm freqs} = 16$ Fourier frequencies, giving
a feature dimension of $33$, processed by an MLP with depth $4$
and hidden width $256$ to produce $\mathbf{t}(\mathbf{x}) \in
\mathbb{R}^p$. The operator output is formed as
\begin{equation}
G^\dagger_\theta(\mathbf{d})(\mathbf{x}) =
\sum_{k=1}^p b_k(\mathbf{d})\, t_k(\mathbf{x}),
\label{eq:deeponet}
\end{equation}
with basis dimension $p = 128$. Four branch and trunk combinations
were evaluated on the Wing and subsurface interface problems: an MLP
branch with an MLP trunk, an MLP branch with a Fourier trunk, a CNN
branch with an MLP trunk, and a CNN branch with a Fourier trunk. The
MLP branch with Fourier trunk combination produced the best
reconstruction quality and is used throughout all three problems for
consistency.

All three architectures use GELU activations, batch normalization, $L_2$ weight regularization with strength $\lambda = 10^{-4}$, dropout with rate $0.2$, and a softplus output activation for the Wing and subsurface interface problems to ensure $\hat{\mathbf{m}} \geq \mathbf{0}$. Training pairs are generated by drawing $\mathbf{m}_i \sim \mu$ and computing $\mathbf{d}_i =
G(\mathbf{m}_i) + \mbox{\boldmath$\varepsilon$} _i$,
$\mbox{\boldmath$\varepsilon$}_i \sim \mathcal{N}(\mathbf{0}, C_d)$, with $\mathbf{m}_i$ and $\mbox{\boldmath$\varepsilon$}_i$ independent.

All networks minimize the empirical risk $J_K(\theta)$ defined in Section~3. Physics-informed networks (PINNs) minimize the augmented objective $\mathcal{R}_\alpha$ with $\alpha = 1$, balancing the physics residual against prior-induced and architectural regularization. Networks are optimized with the Adaptive Moment Estimation with Weight Decay (AdamW) optimizer with learning rate $10^{-3}$, weight decay $10^{-4}$, and gradient clipping norm $1.0$. Training uses an $85\%/15\%$ train-validation split, batch size $256$, and a maximum of $1000$ epochs. Early stopping with patience $100$ and learning rate reduction on plateau with factor $0.5$ and patience $15$ are applied throughout. Input and output vectors are standardized to zero mean and unit variance.

\subsection{Baseline Methods and Evaluation Metric}

The MAP estimate under a Gaussian prior serves as the classical
regularization baseline throughout. For the Wing problem, where
$G$ is linear, the MAP estimate is computed in closed form as
established in Section~2. For the subsurface interface and seismic
tomography problems, where $G(\mathbf{m})$ is nonlinear, the MAP
estimate is approximated by the Gauss-Newton iteration initialized
at $\mathbf{m}^{(0)} = \mathbf{m}_0$ and run for $10$ steps. A
separate baseline is computed under each covariance structure. The
prior sample mean $\bar{\mathbf{m}} = \frac{1}{K}\sum_{i=1}^K
\mathbf{m}_i$ is included as a second baseline to confirm that the
learned operator extracts meaningful information from
$\mathbf{d}_{\rm obs}$ rather than collapsing to the prior mean.

The evaluation metric for all experiments is the relative $\ell_2$
error
\begin{equation}
e_{\rm rel} = \frac{\|\hat{\mathbf{m}} - \mathbf{m}_{\rm true}\|_2}
{\|\mathbf{m}_{\rm true}\|_2},
\label{eq:rel_l2}
\end{equation}
where $\hat{\mathbf{m}}$ is the network prediction or baseline
estimate. For the Wing and subsurface interface problems, $e_{\rm
rel}$ is computed in parameter space. For the seismic tomography
problem, $e_{\rm rel}$ is computed in velocity space by substituting
$\hat{v} = 1/\hat{m}$ and $v_{\rm true} = 1/m_{\rm true}$, since
velocity is the physically interpretable quantity in seismic
tomography.


\subsection{Wing Problem}
\label{sec:wing}
Results are obtained under the correlated covariance with $\Delta = 0.02$, unless otherwise stated.

Figure~\ref{fig:wing_best} shows the best-performing architecture reconstruction $\hat{\mathbf{m}} = G^\dagger_\theta(\mathbf{d}_{\rm obs})$ per prior family, trained by minimizing the standard empirical risk $J_K(\theta)$. The sample mean is approximately zero, consistent with all four priors
being centered at $\mathbf{m}_0 = \mathbf{0}$, yet all learned operators recover the pulse location and a positive amplitude, consistent with convergence to $E_\mu[\mathbf{m} \mid \mathbf{d}]$ established in Theorem~1. Under the Gaussian, Laplace, and uniform
priors, the best-performing architectures produce smooth
reconstructions centered near the true pulse region with nearly indistinguishable results. The total variation prior, by contrast, yields substantially sharper transitions and a higher peak amplitude.

Only the TV prior meaningfully recovers the defining features of $\mathbf{m}_{\rm true}$, namely the pulse location, the discontinuous transitions at $t = 1/3$ and $t = 2/3$, and the near-zero background outside the pulse. The TV CNN recovers a peak amplitude of approximately $1.15$ and begins to resolve the boundary
transitions, though they remain more diffuse than the true discontinuities due to the smoothing action of $G$. The Gaussian, Laplace, and uniform priors produce reconstructions with peak amplitudes reduced to approximately $55$--$60\%$ of the
true value, with transitions entirely unresolved and results indistinguishable from the MAP estimate ($e_{\rm rel} = 0.631$).

\begin{figure}[htbp]
\centering
\includegraphics[width=\textwidth]{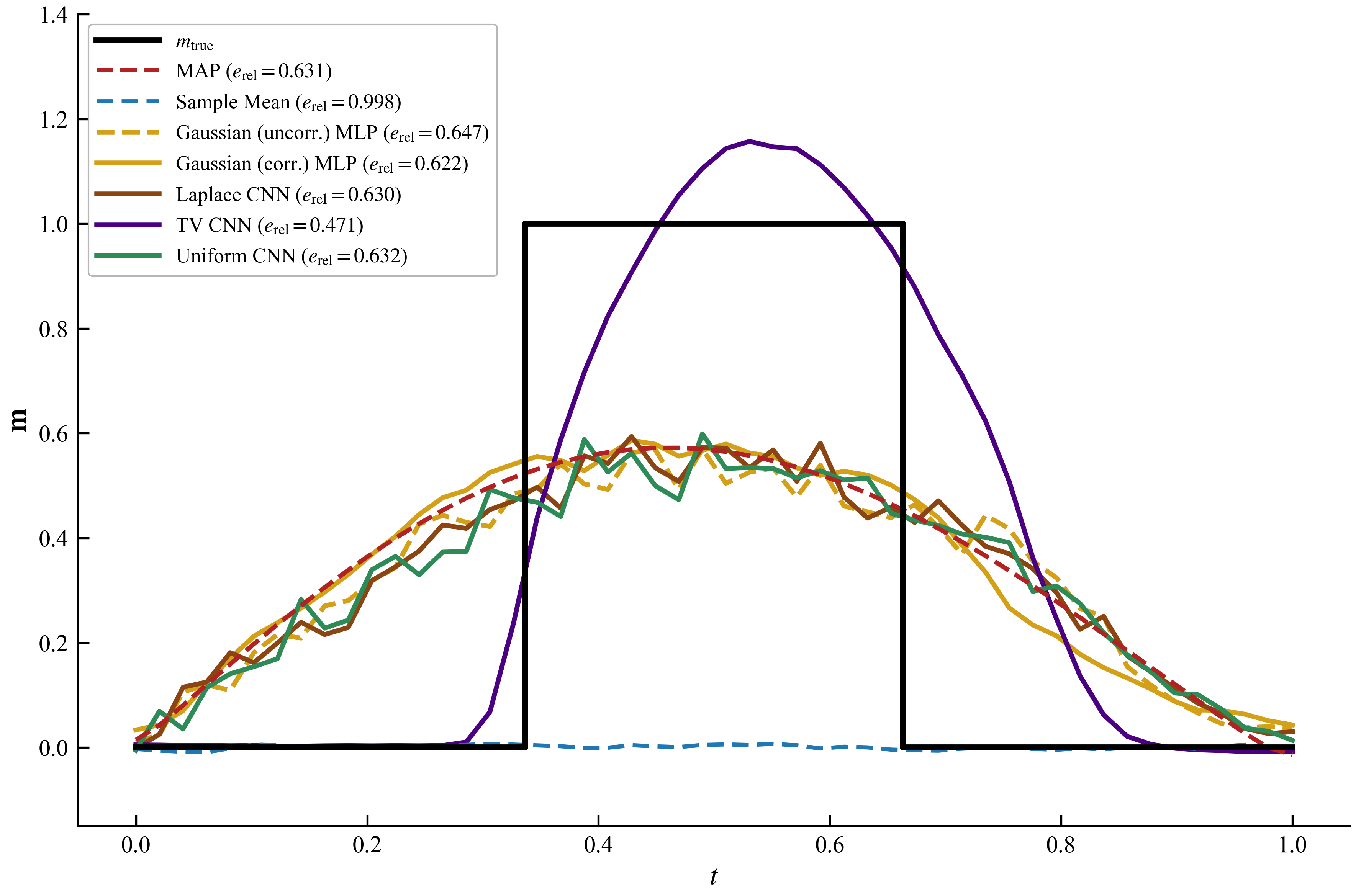}
\caption{Wing problem: best-performing architecture per prior family
under the correlated covariance (lowest $e_{\rm rel}$), with the
Gaussian prior repeated under the uncorrelated covariance for
comparison. The true solution $\mathbf{m}_{\rm true}$ (solid black), MAP estimate $\hat{\mathbf{m}}_{\rm MAP}$ (grey dashed) and prior sample mean $\bar{\mathbf{m}}$ (blue dashed) are included for
comparison. Relative $\ell_2$ errors are reported in
the legend.}
\label{fig:wing_best}
\end{figure}

The near-equivalence of the Gaussian, Laplace, and uniform reconstructions indicates that architectural regularization dominates when the prior distributional shape is poorly matched to $\mathbf m _{\rm true} $. The uniform prior makes this explicit: since it encodes no spatial structure, any regularization in its reconstruction must arise entirely from the network and training procedure, yet it achieves $e_{\rm rel} = 0.632$, indistinguishable from the Gaussian ($0.622$) and Laplace ($0.630$) priors. The TV prior is the only family whose distributional shape competes with this architectural bias. By promoting sparsity of discrete
gradients $|m_{j+1} - m_j|$, it concentrates probability mass on piecewise-constant functions and biases $E_\mu[\mathbf{m} \mid \mathbf{d}]$ toward sharper transitions, partially recovering the discontinuous structure of $\mathbf{m}_{\rm true}$ even where the data provide limited constraint.

For the Gaussian prior, the uncorrelated covariance has $e_{\rm rel} = 0.647$, slightly worse than its correlated counterpart at $0.622$, with no qualitative change in the reconstruction. The modest improvement under the correlated covariance derives from the additional structural information it encodes rather than from any difference in marginal variance.

These results confirm that the learning framework provides no advantage over the classical baseline when the prior is poorly matched to $\mathbf{m}_{\rm true}$, and that prior-structure alignment is the dominant factor of reconstruction quality.

Figure~\ref{fig:wing_pinn} compares the best-performing NN and PINN reconstruction per prior family, obtained by minimizing $J_K(\theta)$ and $\mathcal{R}_\alpha$ with $\alpha = 1$
respectively.

Under the Gaussian, Laplace, and uniform priors, the physics-informed objective produces reconstructions effectively unchanged from those obtained under the standard empirical risk, with no improvement in peak amplitude or boundary resolution. The physics residual enforces data consistency that the architectural regularization has effectively already achieved, producing no meaningful change in the learned operator. Under the TV prior, $e_{\rm rel}$ decreases from $0.471$ to $0.448$, with a corresponding sharpening of the predicted transitions near $t = 1/3$ and $t = 2/3$. 

The contribution of $\mathcal{R}_\alpha$ is therefore determined by the alignment between $\mu$ and the structure of $\mathbf{m}_{\rm true}$. The augmented objective acts as a refinement when the prior is well matched, but yields no meaningful benefit when it is mismatched.

\begin{figure}[htbp]
\centering
\includegraphics[width=\textwidth]{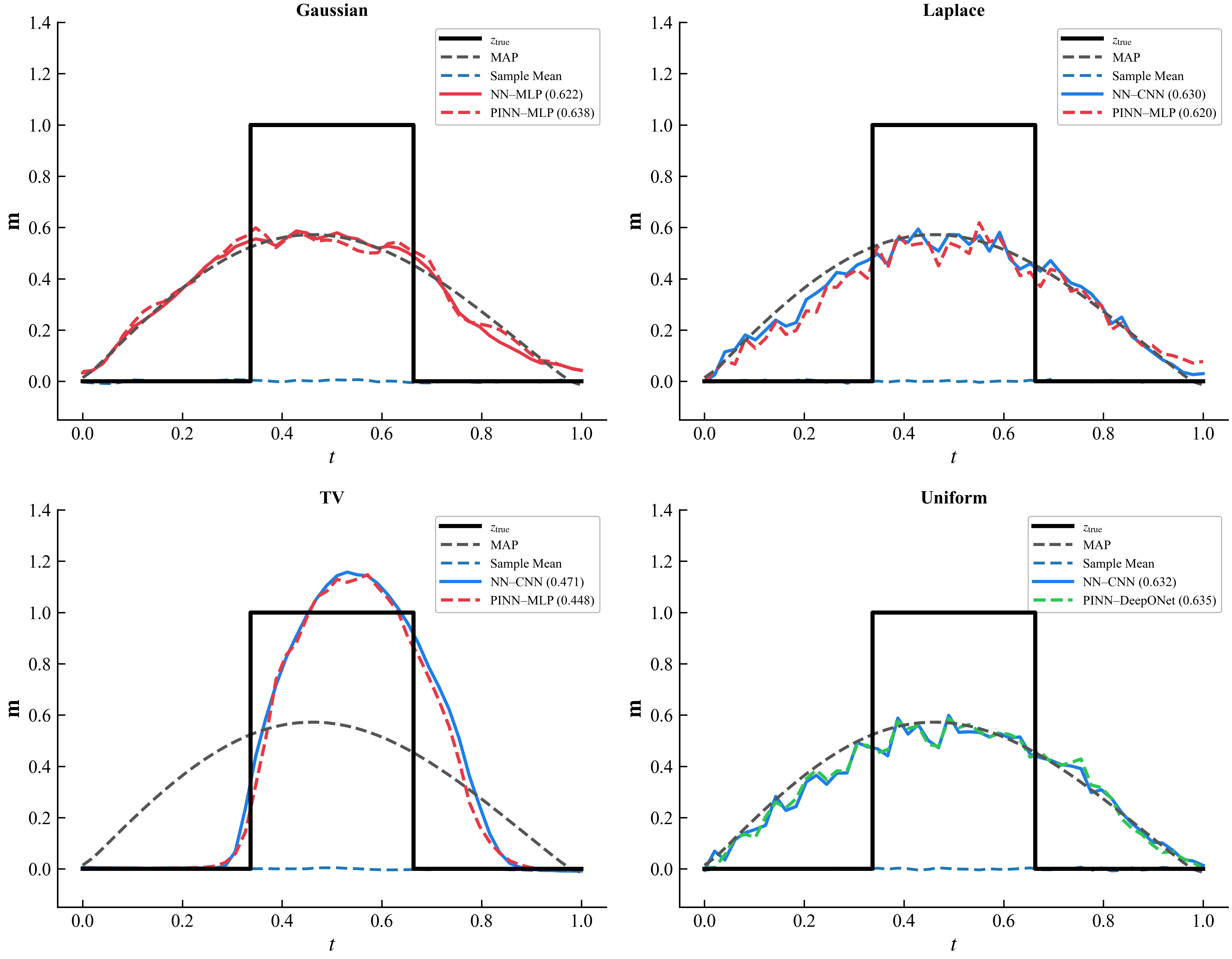}
\caption{Wing problem: best-performing NN (solid curves, lowest
$e_{\rm rel}$) and best-performing PINN (dashed curves, lowest
$e_{\rm rel}$) per prior family, obtained by minimizing $J_K(\theta)$
and $\mathcal{R}_\alpha$ with $\alpha = 1$ respectively. The true
solution $\mathbf{m}_{\rm true}$ (solid black), MAP estimate
$\hat{\mathbf{m}}_{\rm MAP}$ (grey dashed), and prior sample mean
$\bar{\mathbf{m}}$ (brown dashed) are included for comparison.
Relative $\ell_2$ errors are reported in the panel legends.}
\label{fig:wing_pinn}
\end{figure}

Table~\ref{tab:wing_metrics} summarizes the relative $\ell_2$ errors for the best-performing architectures under each prior and objective.

\begin{table}[htbp]
\centering
\begin{tabular}{|lcc|}
\hline
Prior & Best NN $e_{\rm rel}$ & Best PINN $e_{\rm rel}$ \\
\hline
Gaussian (correlated)   & $0.622$ (MLP)      & $0.638$ (MLP)      \\
Gaussian (uncorrelated) & $0.647$ (MLP)      & $0.645$ (DeepONet) \\
Laplace                 & $0.630$ (CNN)      & $0.620$ (MLP)      \\
TV                      & $0.471$ (CNN)      & $0.448$ (MLP)      \\
Uniform                 & $0.632$ (CNN)      & $0.635$ (DeepONet) \\
\hline
\end{tabular}
\caption{Relative $\ell_2$ errors $e_{\rm rel}$ for the Wing
problem. The architecture with the lowest $e_{\rm rel}$ is shown
in parentheses. The MAP estimate yields $e_{\rm rel} = 0.631$.}
\label{tab:wing_metrics}
\end{table}

\subsection{Subsurface Interface Inversion Results}
\label{sec:anomaly_results}
Results are obtained under the correlated covariance with $\Delta = 1$ km, unless otherwise stated.

Figure~\ref{fig:anomaly_best} shows the best-performing architecture reconstruction $\hat{\mathbf{m}} = G^\dagger_\theta(\mathbf{d}_{\rm obs})$ per prior family, trained by minimizing $J_K(\theta)$. As in the Wing problem, all learned operators recover the anomaly location and a
positive amplitude rather than collapsing to the prior mean. All reconstructions underestimate the true peak value $z_{\max} = 2.5$ km, yet each prior improves over the Gauss-Newton baseline ($e_{\rm rel} = 0.353$), with the Gaussian prior achieving the largest reduction ($e_{\rm rel} = 0.187$). This stands in contrast to the Wing problem,
where the learned operators under the Gaussian, Laplace, and uniform priors were indistinguishable from the MAP estimate.

\begin{figure}[htbp]
    \centering
    \begin{subfigure}{0.72\textwidth}
        \centering
        \includegraphics[width=\textwidth]{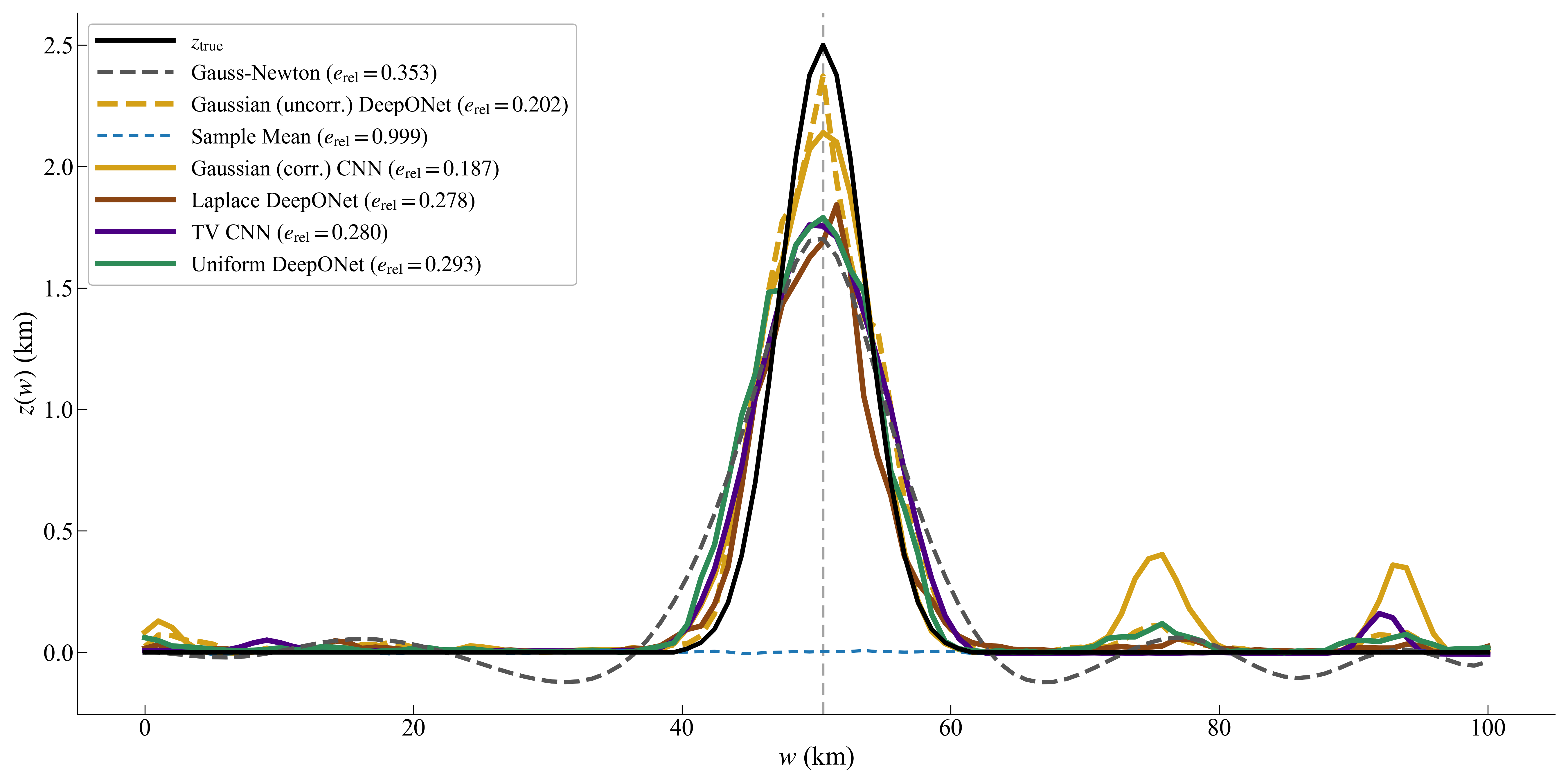}
        \caption{Full reconstruction over $w \in [0, 100]$ km.}
        \label{fig:anomaly_best_full}
    \end{subfigure}
    \hfill
    \begin{subfigure}{0.26\textwidth}
        \centering
        \includegraphics[width=\textwidth]{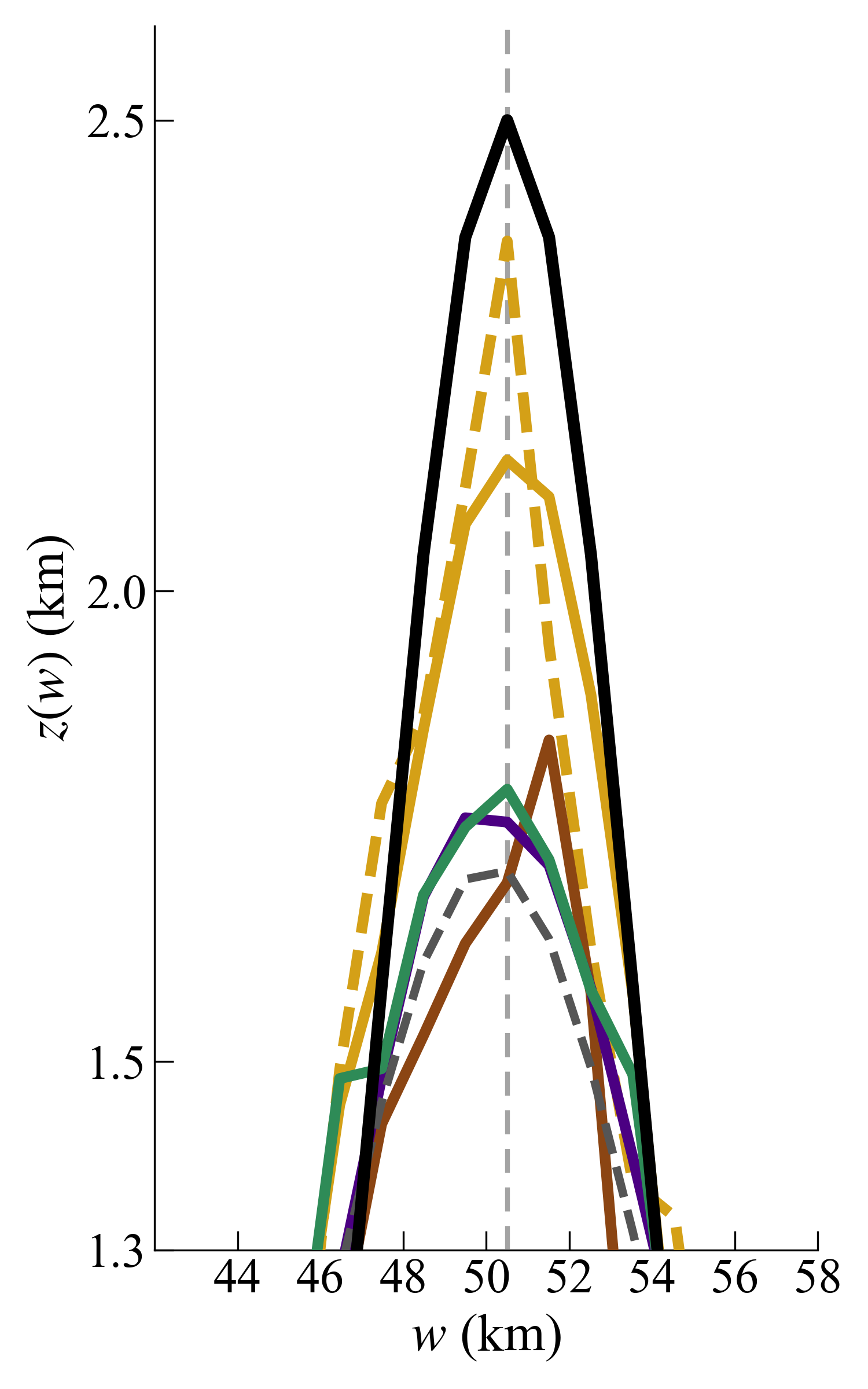}
        \caption{Zoomed view of the peak region
        $w \in [42, 58]$ km.}
        \label{fig:anomaly_best_zoom}
    \end{subfigure}
    \caption{Subsurface interface inversion: best-performing
architecture per prior family under the correlated covariance
(lowest $e_{\rm rel}$), with the Gaussian prior repeated under
the uncorrelated covariance (dashed gold) for comparison. The
true interface $z_{\rm true}$ (solid black), Gauss-Newton
estimate (grey dashed, $e_{\rm rel} = 0.353$), and prior sample
mean (blue dashed) are included for
comparison. The vertical dashed line marks the true peak
location $w_0 = 50.5$ km. Relative $\ell_2$ errors are
reported in the legend of panel~(a).}
    \label{fig:anomaly_best}
\end{figure}

The Gaussian prior correctly places the peak at $w_0 = 50.5$ km and recovers the highest peak amplitude ($\approx 2.1$ km), closely matching the sharp peak of $z_{\rm true}$, though oscillations around zero appear over $w \in [60, 100]$ km. The Laplace prior places the peak at $w \approx 52$ km, a shift of approximately $1.5$ km to the right of $w_0$, with a peak amplitude ($\approx 1.85$ km) and a clean reconstruction outside the anomaly region. The TV prior places the peak near $w_0 = 50.5$ km but produces a broad smooth reconstruction with a peak amplitude ($\approx 1.75$ km). The uniform prior also places the peak at $w_0 = 50.5$ km with a peak amplitude ($\approx 1.75$ km) and
only minor deviations outside the anomaly region.

The Gaussian prior produces the most accurate reconstruction, consistent with its smooth unimodal distributional shape matching that of $z_{\rm true}$. The Laplace prior recovers a competitive peak amplitude but with a shifted peak location, reflecting its $\ell_1$-type regularization. The TV prior produces a broad reconstruction, reflecting the mismatch between its piecewise-constant regularization
character and the smooth unimodal structure of $z_{\rm true}$. The uniform prior, which encodes no distributional shape, still improves over the Gauss-Newton baseline ($e_{\rm rel} = 0.293$), confirming that architectural regularization contributes to reconstruction quality
independent of prior distributional shape. Unlike the Wing problem, the choice of prior family has a measurable effect on reconstruction quality, indicating that prior-structure alignment plays a more substantial role for this nonlinear problem.

Figure~\ref{fig:anomaly_cov}(a) shows the effect of covariance
structure on reconstruction quality for the Gaussian prior. Under the correlated covariance, all three architectures recover the anomaly at $w_0 = 50.5$ km with a narrow peak profile. CNN attains the lowest error ($e_{\rm rel} = 0.187$), followed by DeepONet ($0.202$) and MLP ($0.242$), all improving over the Gauss-Newton baseline ($e_{\rm rel}= 0.353$). For the uncorrelated covariance, only DeepONet recovers the anomaly with a smooth peaked profile ($e_{\rm rel} = 0.202$), while MLP and CNN produce broad distorted reconstructions ($e_{\rm rel} = 0.608$ and $0.423$ respectively), both worse than the Gauss-Newton baseline ($e_{\rm rel} = 0.446$). These results confirm that spatial correlation in the prior covariance is necessary for MLP and CNN to recover the anomaly accurately. DeepONet, through its Fourier trunk, provides sufficient spectral inductive bias to reconstruct the smooth Gaussian-shaped anomaly without spatial correlation in the prior covariance.

\begin{figure}[htbp]
    \centering
    \begin{subfigure}{\textwidth}
        \centering
        \includegraphics[width=\textwidth]{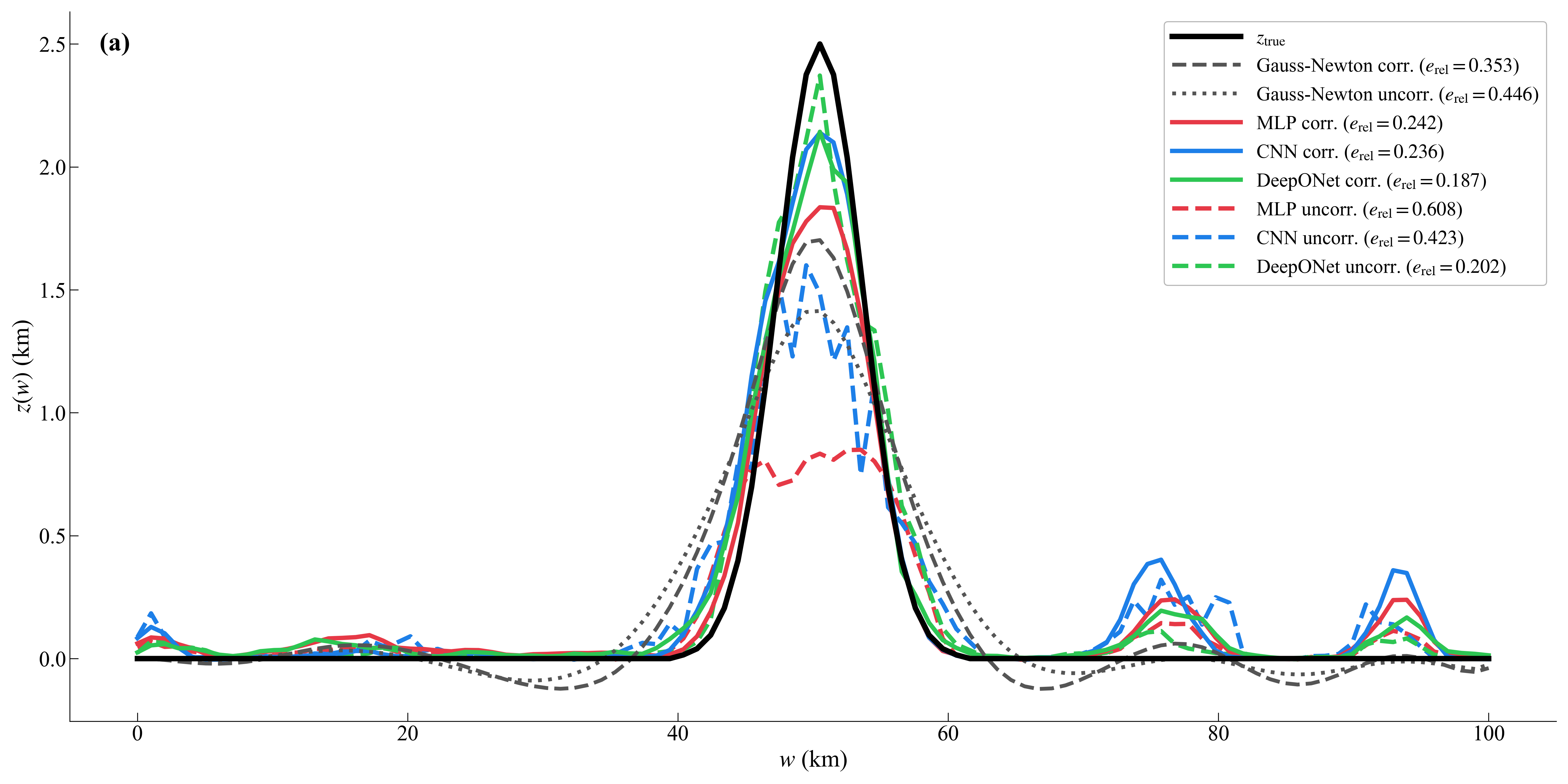}
        \caption{NN reconstructions for all three
        architectures under the correlated covariance
        (solid) and scaled identity covariance (dashed).}
        \label{fig:anomaly_cov_nn}
    \end{subfigure}
    \vspace{1em}
    \begin{subfigure}{\textwidth}
        \centering
        \includegraphics[width=\textwidth]{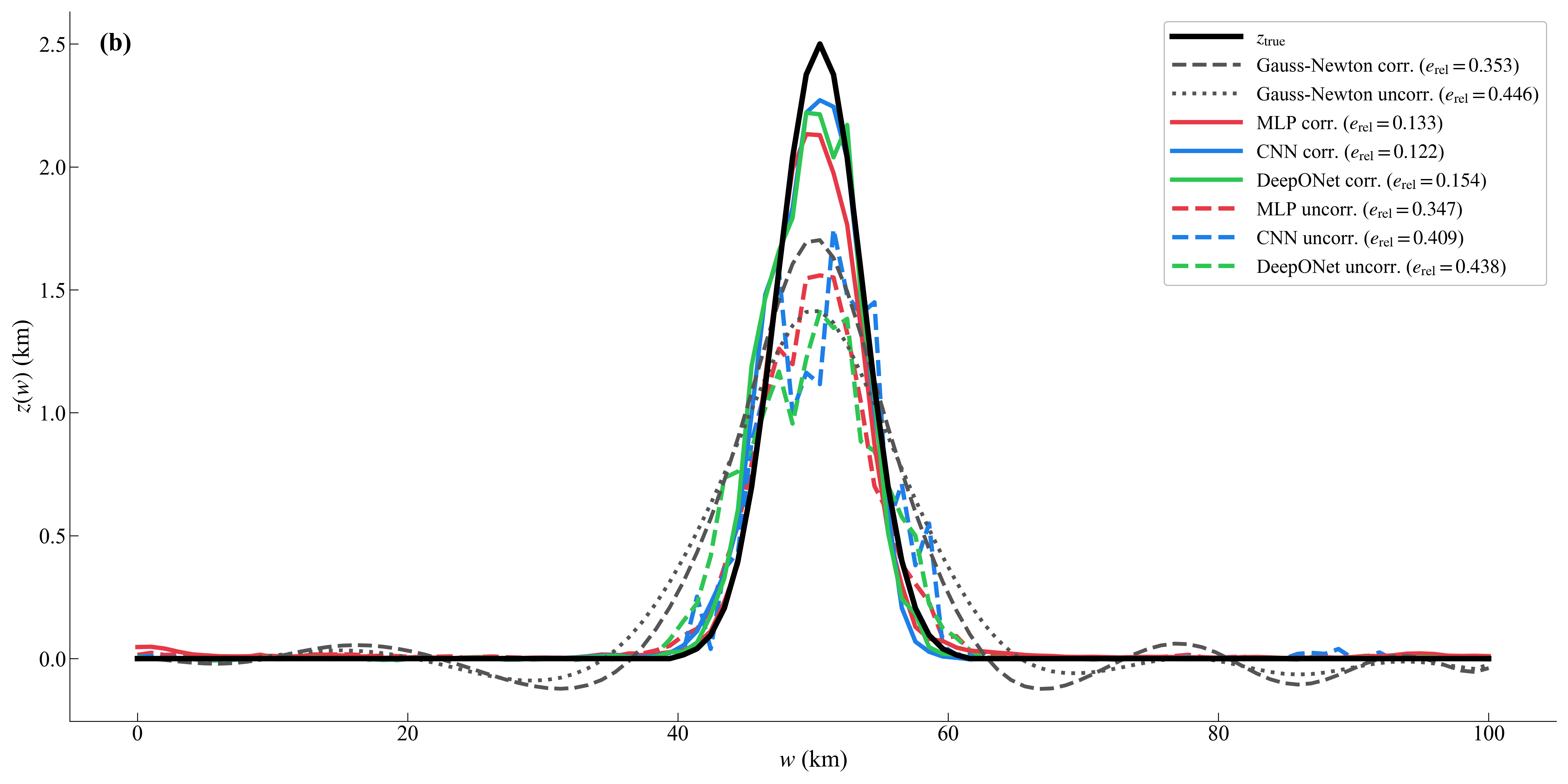}
        \caption{PINN reconstructions for all three
        architectures under the correlated covariance
        (solid) and scaled identity covariance (dashed).}
        \label{fig:anomaly_cov_pinn}
    \end{subfigure}
    \caption{Subsurface interface inversion: effect of
covariance structure on reconstruction quality for the
Gaussian prior. The true interface $z_{\rm true}$ (solid
black) and Gauss-Newton estimates under the correlated
covariance (grey dashed, $e_{\rm rel} = 0.353$) and scaled
identity covariance (grey dotted, $e_{\rm rel} = 0.446$)
are included for comparison. Relative $\ell_2$ errors are
reported in the legend.}
    \label{fig:anomaly_cov}
\end{figure}

Figure~\ref{fig:anomaly_pinn} compares the best-performing NN and PINN reconstruction per prior family, obtained by minimizing $J_K(\theta)$ and $\mathcal{R}_\alpha$ with $\alpha = 1$. Across all priors, the physics-informed objective suppresses oscillations outside the anomaly
region and produces cleaner reconstructions. The magnitude of improvement depends on the alignment between $\mu$ and the structure of $z_{\rm true}$. For the Gaussian prior, the PINN-CNN attains a higher peak amplitude and removes oscillations present in the NN reconstruction ($e_{\rm rel} = 0.122$ vs $0.187$). The Laplace prior shows a similar effect, with the PINN correcting the peak shift and increasing amplitude ($e_{\rm
rel} = 0.205$ vs $0.278$). For the TV prior, the improvement is modest, with slight sharpening of the peak and further suppression of background oscillations ($e_{\rm rel} = 0.245$ vs $0.280$). The uniform prior shows no improvement ($e_{\rm rel} = 0.350$ vs $0.293$), indicating that when prior-induced regularization is absent, the physics residual alone is insufficient to produce an accurate reconstruction.

\begin{figure}[htbp]
    \centering
    \includegraphics[width=\textwidth]{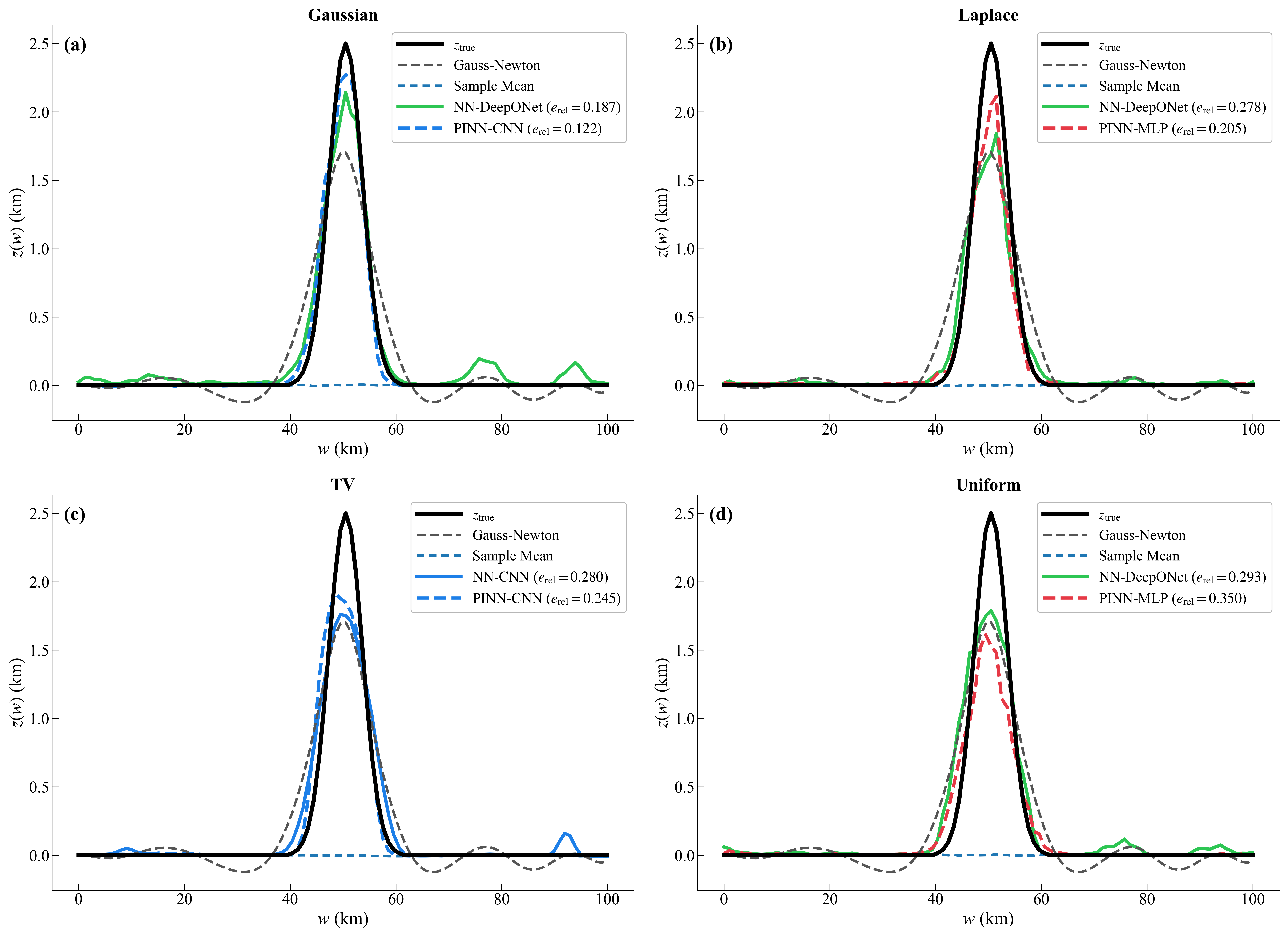}
    \caption{Subsurface interface inversion: best-performing NN (solid curves, lowest $e_{\rm rel}$) and best-performing PINN (dashed curves, lowest $e_{\rm rel}$) per prior family. The true interface $z_{\rm true}$ (solid black), Gauss-Newton estimate (grey dashed), and prior sample mean (brown dashed) are included for comparison. Relative $\ell_2$ errors are reported in the legend of each panel.}
    \label{fig:anomaly_pinn}
\end{figure}

Figure~\ref{fig:anomaly_cov}(b) shows that the PINN is more sensitive to covariance structure than the NN. Under the uncorrelated covariance, all three architectures fail to recover the peak accurately, with $e_{\rm rel}$ values of $0.347$, $0.409$, and $0.438$ for MLP, CNN, and DeepONet respectively, all worse than the Gauss-Newton baseline
($e_{\rm rel} = 0.446$). These results confirm that the contribution of $\mathcal{R}_\alpha$ is determined by prior-structure alignment. The augmented objective acts as a refinement when the prior is well matched to $z_{\rm true}$, but yields no meaningful benefit when it is mismatched.


\subsection{Cross-Well Seismic Traveltime Tomography}
\label{sec:tomo}

Results are obtained under the correlated covariance with
$\Delta = 400$ m, unless otherwise stated.

Figure~\ref{fig:tomo_true_gn} shows the true velocity field alongside the Gauss-Newton reconstructions under both covariance structures. Under the correlated covariance, the slow anomaly
is recovered as a broad diffuse region in the upper left, while the fast anomaly in the lower right is only weakly resolved ($e_{\rm rel} = 0.027$). The scaled identity covariance resolves both anomalies more distinctly, with a more compact slow anomaly
and a clearly visible fast anomaly ($e_{\rm rel} = 0.016$), though an additional bright feature appears near the slow anomaly that is not present in $v_{\rm true}$.

\begin{figure}[htbp]
\centering
\includegraphics[width=\textwidth]{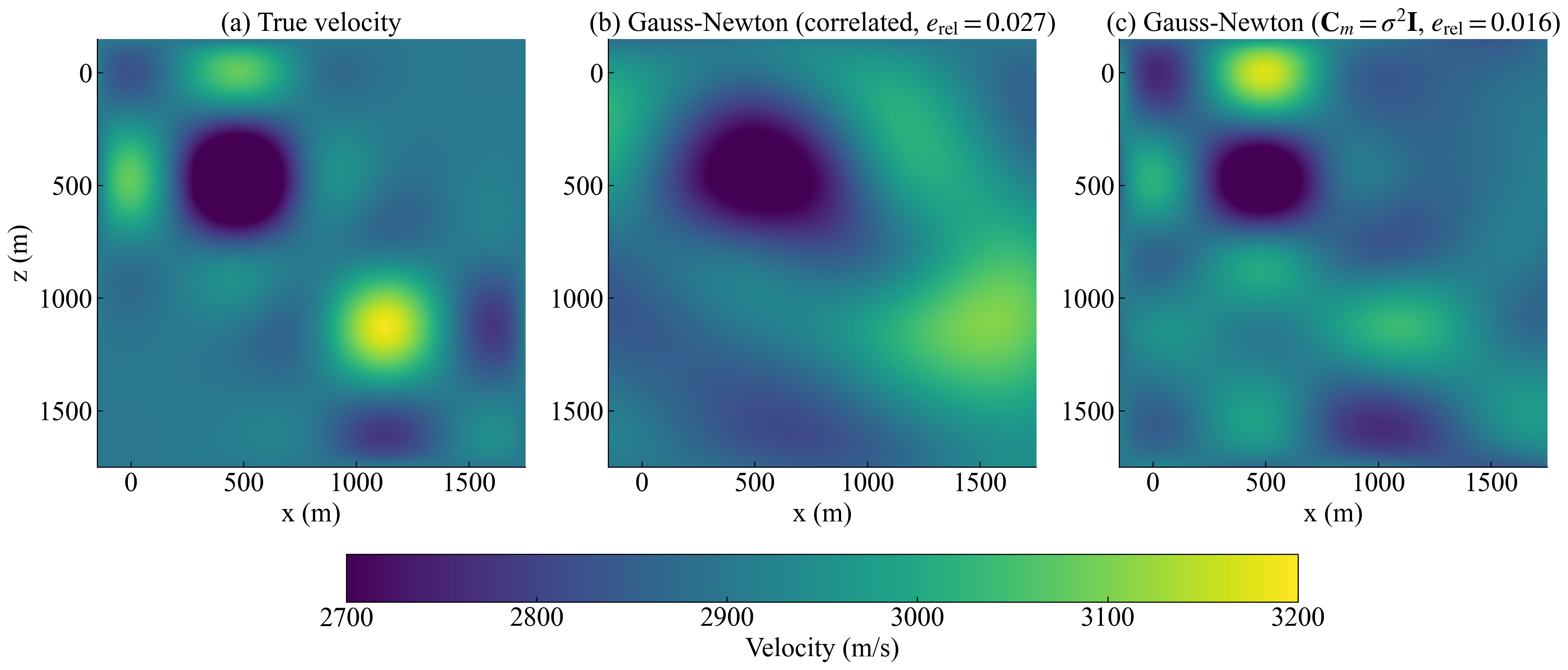}
\caption{Cross-well seismic tomography: (a) true velocity field and Gauss-Newton reconstructions under (b) the correlated covariance ($\Delta = 400$ m, $e_{\rm rel} = 0.027$) and (c) the scaled identity covariance ($C_m = \sigma^2 \mathbf{I}$, $e_{\rm rel} = 0.016$).}
\label{fig:tomo_true_gn}
\end{figure}

Figure~\ref{fig:tomo_best} shows the best-performing architecture reconstruction $\hat{\mathbf{m}} = G^\dagger_\theta(\mathbf{d}_{\rm
obs})$ per prior family, trained by minimizing $J_K(\theta)$. Across all prior families, the learned operators recover the location of the slow anomaly in the upper left, with reconstruction quality varying substantially across prior families.

\begin{figure}[htbp]
\centering
\includegraphics[width=\textwidth]{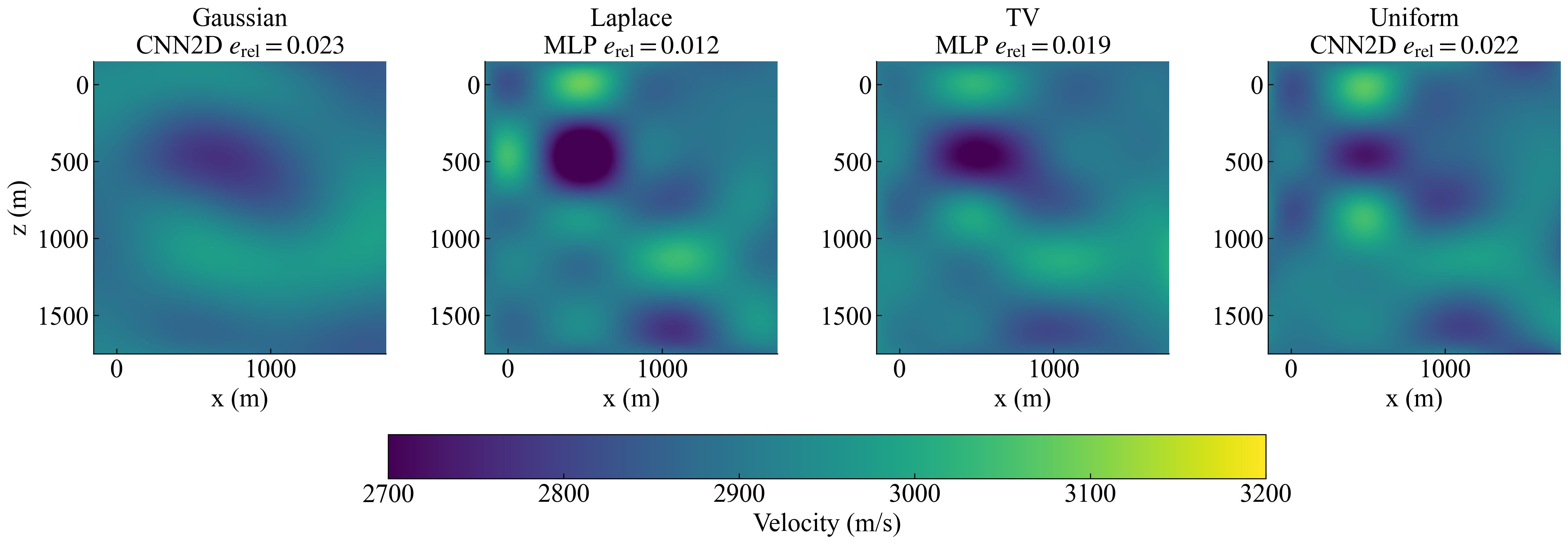}
\caption{Cross-well seismic tomography: best-performing architecture per prior family (lowest $e_{\rm rel}$) under the correlated covariance ($\Delta = 400$ m).}
\label{fig:tomo_best}
\end{figure}

The Laplace prior produces the sharpest anomaly boundaries and highest contrast among the four families, with the slow anomaly appearing as a compact well-localized feature at the
correct position and the fast anomaly clearly resolved ($e_{\rm rel} = 0.012$). The TV prior recovers both anomalies with moderate contrast and moderately sharp boundaries ($e_{\rm rel} = 0.019$), with the slow anomaly more spatially
diffuse than in the Laplace case and the fast anomaly only weakly resolved. The Gaussian prior recovers the slow anomaly as a broad diffuse region but fails to resolve the fast anomaly ($e_{\rm rel} = 0.023$). The uniform prior recovers both anomaly locations with intermediate contrast ($e_{\rm
rel} = 0.022$), but exhibits additional fast and slow features not present in $v_{\rm true}$.

The improved reconstruction quality of the Laplace prior is consistent with its $\ell_1$-type regularization promoting sparse deviations from $\mathbf{m}_0$. The true anomalies
are spatially compact perturbations against a uniform background, consistent with the localized deviations from $\mathbf{m}_0$ induced by the Laplace prior, and this prior-structure alignment accounts for its improved reconstruction quality relative to the other three families. The TV prior yields an intermediate reconstruction reflecting
its piecewise-constant regularization character, which promotes sharp transitions in the gradient field but does not favor the smooth Gaussian spatial extent of the true anomalies. The uniform prior, which encodes no spatial structure, recovers both anomaly locations but introduces additional features not present in $v_{\rm true}$. The
Laplace, TV, and uniform priors all improve over the
Gauss-Newton baseline ($e_{\rm rel} = 0.027$), while the Gaussian prior produces a reconstruction qualitatively similar to Gauss-Newton ($e_{\rm rel} = 0.023$), consistent with the smooth solutions induced by both the Gaussian prior and the correlated covariance, which together suppress the
localized high-contrast features needed to resolve the true anomalies.

Figure~\ref{fig:tomo_cov} examines the effect of covariance structure on reconstruction quality for the Gaussian prior. Under the scaled identity covariance, both the learned operator and the Gauss-Newton estimate produce more localized reconstructions, with a more compact slow anomaly and a clearly resolved fast anomaly. The best NN reconstruction attains $e_{\rm rel} = 0.016$, matching the Gauss-Newton baseline under the same covariance. For the Laplace, TV, and uniform priors, only the diagonal of $C_m$ enters the
construction of the component-wise scale parameters, so both covariance choices produce nearly identical reconstructions. Even under the scaled identity covariance, where the Gaussian prior reconstruction improves substantially, it does not match the reconstruction quality
achieved by the Laplace prior under either covariance structure, confirming that prior distributional shape is the dominant factor of reconstruction quality.

The physics-informed objective $\mathcal{R}_\alpha$ with $\alpha = 1$ produces reconstructions nearly indistinguishable from those obtained under $J_K(\theta)$ across all four prior families, with differences in $e_{\rm rel}$ of at most $0.001$. This indicates that the regularization induced by the training prior dominates 
in this setting, and the physics residual provides only a marginal additional 
correction, consistent with the pattern established in the Wing and subsurface 
interface problems, where the physics-informed objective acts as a refinement 
when the prior is well matched to the true model, but yields no meaningful 
benefit when it is mismatched.

\begin{figure}[htbp]
\centering
\includegraphics[width=\textwidth]{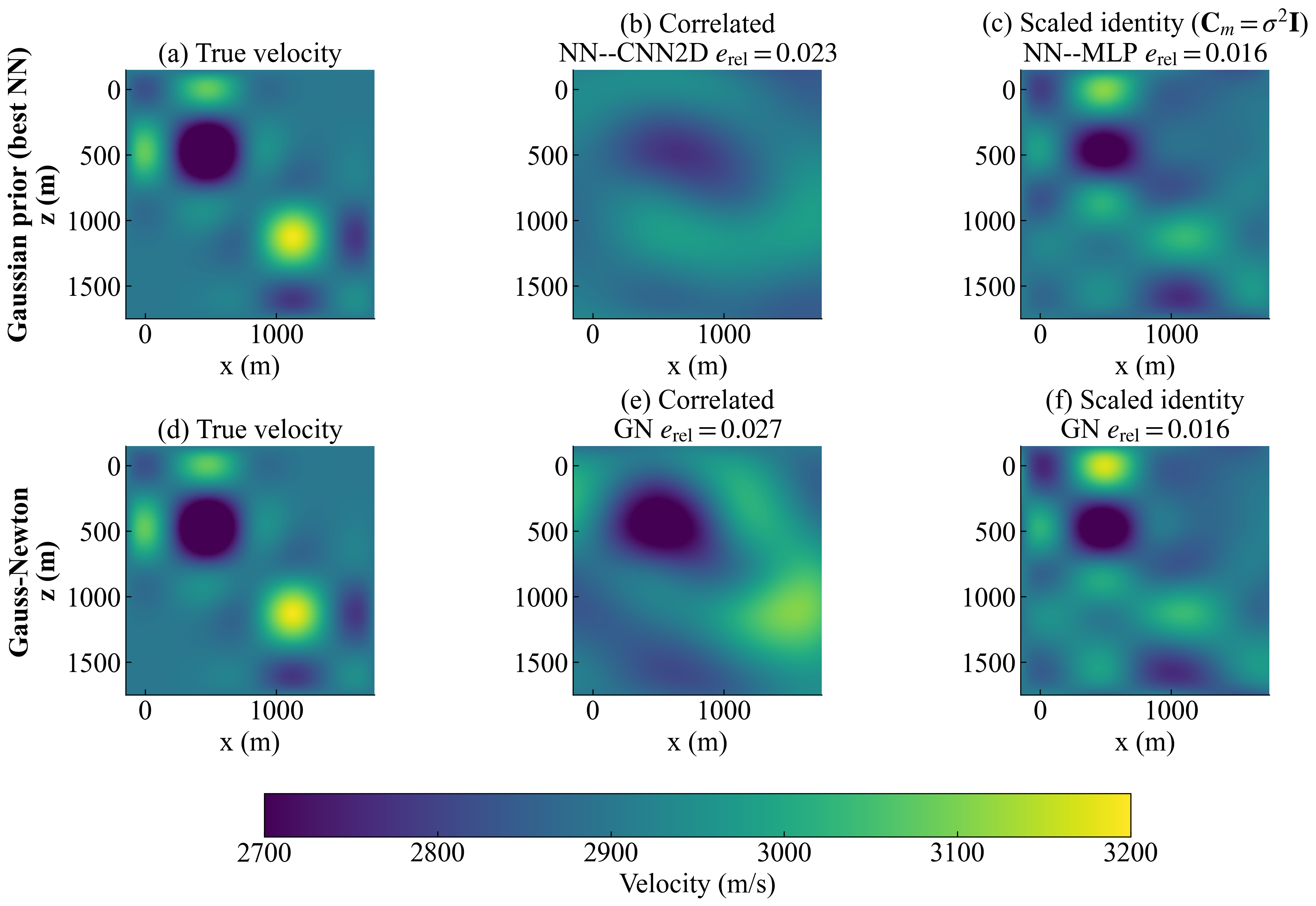}
\caption{Cross-well seismic tomography: effect of covariance structure on reconstruction quality for the Gaussian prior. Top row: best-performing NN reconstruction (lowest $e_{\rm
rel}$) under the correlated covariance and scaled identity covariance. Bottom row: Gauss-Newton reconstruction under both covariance structures. Panels~(a) and~(d) show the true velocity field for orientation. Relative $\ell_2$ errors are reported in each panel title.}
\label{fig:tomo_cov}
\end{figure}

\section{Discussion and Conclusion}
\label{discussion}

The experiments across three inverse problems of increasing complexity reveal a consistent pattern governed by three distinct sources of regularization in $G^\dagger_\theta$. The first is prior-induced regularization, determined by the distributional shape of the sampling measure $\mu$. The second is architectural regularization, reflecting the inductive bias of the network architecture and training procedure. The third is physics-informed regularization, introduced through the augmented objective $\mathcal{R}_\alpha$. Their relative contributions vary systematically across problems, prior families, and architectures.

Among the three, the dominant source is prior-induced regularization. Theorem~1 establishes that $G^\dagger_\theta$ converges to $E_\mu[\mathbf{m} \mid \mathbf{d}]$, so the distributional shape of $\mu$ strongly influences the structural character of the learned solution. The experiments confirm this theoretical connection across all three problems. The TV prior, which implements the discrete total variation functional $\sum_j |m_{j+1} - m_j|$ as the implicit regularizer, produces the sharpest reconstructions for the Wing problem, whose true solution $\mathbf{m}_{\rm true}$ is a discontinuous rectangular pulse. The Gaussian prior, which induces Tikhonov-type
regularization and yields smooth solutions, produces the most accurate reconstructions for the subsurface interface problem, whose true solution is a smooth unimodal anomaly. The Laplace prior, whose $\ell_1$-type regularization promotes sparse deviations from $\mathbf{m}_0$, produces the most accurate reconstructions for the seismic tomography problem, where the true anomalies are spatially compact perturbations against a uniform background, resolving both anomalies with the sharpest boundaries and highest contrast among the four families. In each
case, the prior whose distributional shape most closely matches the structure of $\mathbf{m}_{\rm true}$ yields the most physically accurate reconstruction. This prior-structure alignment emerges as the primary determinant of reconstruction quality across all three problems.

A second, independent source of regularization arises from the inductive bias of the network architecture and training procedure. The uniform prior isolates this contribution: since its components are drawn independently with no spatial structure, any regularization in the resulting reconstructions arises entirely from the network. This is consistent across all three problems, with the uniform prior recovering the pulse location in the Wing problem, the anomaly peak in the subsurface interface problem, and both anomaly locations in the seismic tomography problem. For the nonlinear problems, however, architectural regularization alone is insufficient to produce physically accurate reconstructions or suppress additional features not present in $\mathbf{m}_{\rm true}$ without appropriate prior constraints, indicating that prior-structure alignment plays a more consequential role as problem complexity increases.

Finally, physics-informed regularization, introduced through $\mathcal{R}_\alpha$ with $\alpha = 1$, operates alongside the first two sources. As established in Theorem~2, minimizing $\mathcal{R}_\alpha$ is equivalent to learning $E_{\mu_\alpha}[\mathbf{m} \mid \mathbf{d}]$ under the tilted measure $\mu_\alpha \propto e^{-\alpha\|G(\mathbf{m}) - \mathbf{d}\|^2} \mu$, which concentrates probability mass on models consistent with the forward operator $G$. The contribution of $\mathcal{R}_\alpha$ is determined by prior-structure alignment. When the prior is well matched to $\mathbf{m}_{\rm true}$, the physics residual acts as a refinement, suppressing oscillations and improving reconstruction accuracy, as observed for the TV prior in the Wing problem and the Gaussian and Laplace priors in the subsurface interface problem. When the prior is poorly matched, the physics residual provides no meaningful benefit, as observed for the uniform prior in the subsurface interface problem and across all prior families in the seismic tomography problem. The augmented objective $\mathcal{R}_\alpha$ is therefore most effective as a refinement of a well-matched prior, rather than as a mechanism for compensating for
prior-structure mismatch.

The covariance structure of the Gaussian prior also influences reconstruction 
quality. For example, spatial correlation (equation~\ref{eq:cov}) is necessary for 
MLP and CNN to recover the smooth Gaussian-shaped anomaly in the subsurface 
interface problem, while DeepONet, through its Fourier trunk, provides sufficient 
spectral inductive bias to reconstruct the anomaly without it. For the seismic tomography problem, the correlated covariance reinforces the smooth solutions induced by the Gaussian prior, 
limiting qualitative improvement over the Gauss-Newton baseline. In both cases, however, 
prior distributional shape remains the dominant determinant of reconstruction quality when the two sources are compared directly.

Together, these results establish a clear design hierarchy
for neural network-based inversion. The learned operator
$G^\dagger_\theta$ approximates the conditional expectation $E_\mu[\mathbf{m} \mid \mathbf{d}]$ rather than the posterior mode, and therefore minimizes mean-square error. This provides a principled explanation for the observed improvement over MAP estimates, which is most pronounced for the nonlinear
problems and non-Gaussian prior families where the posterior is not unimodal. Prior selection is the primary design consideration and should be guided by the structural character of $\mathbf{m}_{\rm true}$. A mismatched prior degrades reconstruction quality in ways that neither architectural nor physics-informed regularization can fully correct. The physics-informed objective $\mathcal{R}_\alpha$ is most effective as a refinement when the prior is well matched, enforcing data consistency at low additional cost, and should not substitute for careful prior selection.

A direct consequence of the results presented here concerns the role of the sampling distribution in learned inversion. In many scientific inverse problems, observational data is insufficient to train $G^\dagger_\theta$ directly, and synthetic training
pairs are generated by drawing parameters from a prescribed distribution $\mu$ and evaluating the forward model. The results of this study demonstrate that this choice of $\mu$ is not a neutral computational step but rather a regularization decision with direct consequences for reconstruction quality. Across all three problems, the distributional shape of $\mu$ determined the
structural character of $E_\mu[\mathbf{m} \mid \mathbf{d}]$ more substantially than either the network architecture or $\mathcal{R}_\alpha$. A mismatched $\mu$ produced reconstructions that neither a more expressive architecture nor an augmented physics residual could meaningfully correct. This establishes a concrete design principle: $\mu$ should be chosen to reflect the structural character of $\mathbf{m}_{\rm true}$, since in the infinite-data limit, as established in Theorem~1, the sampling distribution and the regularization functional are equivalent.

The present study opens several directions for future
investigation. The weighting parameter $\alpha = 1$ is fixed throughout, and its sensitivity across prior families, architectures, and problem settings has not been systematically investigated. The three test problems are relatively low-dimensional, with $N$ ranging from $49$ to $100$, and the extent to which the three-source regularization framework extends to higher-dimensional settings remains an open question. The study considers three architectures and four prior families, and
the prior-structure alignment principle may not extend uniformly to broader classes of architectures with stronger or more specialized inductive biases, such as physics-tailored neural operators. More broadly, the connection established here between the choice of $\mu$ and the regularization character of $G^\dagger_\theta$ provides a principled framework for prior design in neural network-based inversion, with direct implications for scientific applications where structural knowledge of $\mathbf{m}_{\rm true}$ is available.

\flushleft{\textbf{References}} \\

\bibliographystyle{unsrt}
\bibliography{ref.bib}
\end{document}